\documentclass[12pt]{amsart}
\usepackage{amscd,amssymb,epsfig}
\usepackage{graphics}

\input epsf
\newtheorem{theorem}{Theorem}[section]
\newtheorem{lemma}[theorem]{Lemma}
\newtheorem{proposition}[theorem]{Proposition}
\newtheorem{corollary}[theorem]{Corollary}
\theoremstyle{definition}
\newtheorem{definition}[theorem]{Definition}

\theoremstyle{remark}

\numberwithin{figure}{section}
\numberwithin{table}{section}

\def\ra{{\rightarrow}}

\def\lra{{\leftrightarrow}}
\def\hra{{\hookrightarrow}}
\newcommand{\xra}{\xrightarrow}

\def\omp{{\omega'_h}}

\def\Poin{{Poincar\'e }}

\def\bZ{{\mathbb Z}}
\def\bQ{{\mathbb Q}}

\def\S1g{{\Sigma_{g,1}}}
\def\M1g{{\mathcal M}_{g,1}}
\def\I1g{{\mathcal I}_{g,1}}

\def\tG{\tau^{GG'}}
\def\btG{\bar \tau^{GG'}}

\def\El{\mathcal L}
\def\half{{\frac 1 2}}
\def\sixth{{\frac 1 6 }}
\newcommand{\mc}{\mathcal}

\newcommand{\mb}{\mathbf}
\newcommand{\ThG}[3]{\theta_{#1}^{#2}(\mb{#3})}
\newcommand{\f}{\mb{f}}
\newcommand\Hom{\operatorname{Hom}}

\newcommand\Aut{\operatorname{Aut}}

\newcommand\Ker{\operatorname{Ker}}

\newcommand\im{\operatorname{Im}}

\newcommand\Diff{\operatorname{Diff}}

\catcode`\@=\active 

\begin{document}

\title[Canonical lifts of the Johnson homomorphisms]
{Canonical lifts of the Johnson homomorphisms to the Torelli groupoid}
\author{Alex James Bene}
\address{Department of Mathematics\\
University of Southern California\\
Los Angeles, CA 90089\\
USA}
\email{bene{\char'100}usc.edu}
\author{Nariya  Kawazumi}
\address{Department of Mathematical Sciences\\
University of Tokyo \\Komaba, Tokyo 153-8914\\
Japan}
\email{kawazumi{\char'100}ms.u-tokyo.ac.jp}
\author{R. C.  Penner}
\address{Departments of Mathematics and Physics/Astronomy\\
University of Southern California\\
Los Angeles, CA 90089\\
USA}
\email{rpenner{\char'100}math.usc.edu}

\thanks{The first author thanks Dylan Thurston for useful discussions, 
and the latter two authors thank CTQM and Aarhus University for kind hospitality}

\begin{abstract}
We prove that every trivalent marked bordered fatgraph comes equipped with a canonical generalized Magnus expansion in the sense of Kawazumi.   This  Magnus expansion is used to give canonical lifts
of the higher Johnson homomorphisms $\tau_m$, for $m\geq 1$, to the Torelli groupoid, and we provide a recursive combinatorial formula for tensor representatives of these lifts.  In particular,  we give an explicit 
1-cocycle in the dual fatgraph complex which lifts $\tau_2$ and thus answer  affirmatively  a question of  Morita-Penner.  To illustrate our techniques for calculating higher Johnson homomorphisms in general, we give explicit examples calculating $\tau_m$, for $m\leq 3$.
\end{abstract}

\maketitle

\section{Introduction}

Let $\M1g$ denote the mapping class group of a  genus $g>0$ surface $\S1g$ with one boundary component, and let
$\I1g$ denote its classical Torelli group, i.e., the subgroup of $\M1g$ acting trivially on the homology of $\S1g$.  In fact, this group is just the first of a series of nested subgroups called the ``higher Torelli groups''  $\M1g[k]$, which serve as successive approximations to $\M1g$.  

Johnson extensively studied the first two of these subgroups and in particular determined that for $g\geq 3$, $\I1g=\M1g[1]$ was  finitely generated \cite{Johnson83} while for $g\geq 2$, $\M1g[2]$ was isomorphic to the subgroup $\mc{K}_{g,1}$ generated by Dehn twists on separating curves \cite{Johnson85}.    Following the work of Sullivan \cite{Sullivan75}, Johnson  \cite{Johnson80,Johnson83a} also  defined certain abelian quotients of all Torelli groups, and these quotient maps $\tau_k$  are now called the ``Johnson homomorphisms''.  Johnson moreover determined the images of the first two homomorphisms (up to torsion), 
thus giving maps $\tau_1: \I1g \ra \Lambda^3 H$ and $\tau_2: \mc{K}_{g,1}\ra \Lambda^2 H\otimes \Lambda^2 H/\!\sim$ where $H=H_1(\S1g,\bZ)$ and the relation $\sim$ is recalled in Section \ref{sect:cochainreps}. 

In \cite{Morita93,Morita96,Morita01},  Morita  explicitly showed that the first and second Johnson homomorphisms $\tau_1$ and $\tau_2$ lift to crossed homomorphisms of the mapping class group $\M1g$ 
\begin{equation}  \label{eq:k1}
[\tilde k_1]\in H^1(\M1g,\Lambda^3 H)
\end{equation}
and 
\[
[\tilde k_2]\in H^1(\M1g, \mathcal H_2 \tilde \times \Lambda^3 H),
\]
where $\mathcal H_2$ is a module of 4-tensors of $H$ equal 
(modulo 2-torsion) to the image of $\tau_2$, 
and  the group structure on $\mathcal H_2 \tilde \times \Lambda^3 H$ is defined by a certain skew-symmetric pairing $\Lambda^3 H \times \Lambda^3 H \ra \mathcal H_2$ (see Section~\ref{sect:cochainreps}). 
   See also \cite{Hain,Kawazumi05,Day} regarding lifting the Johnson homomorphisms to the whole of the mapping class group.

Recently, Morita-Penner \cite{moritapenner} showed that the map \eqref{eq:k1} could be canonically lifted to a 
1-cocycle $j_1$ in the dual fatgraph complex (whose definition is recalled  in the next section),
\[
j_1\in Z_{\M1g}^1(\mathcal {\hat G} _T ,\Lambda^3 H),
\]
with $[j_1]=6[\tilde k_1]$.

Morita-Penner then raised the question as to the existence of an analogous ``groupoid'' lift of $\tau_2$.  In this paper, we give an affirmative answer to this question:

\begin{theorem}\label{t2exists}
There exists a canonically defined $\M1g$-equivariant 1-cocycle  
\[
j_2\in Z_{\M1g}^1(\mathcal {\hat G} _T ,\mathcal H_2 \tilde \times\Lambda^3 H)
\]
which represents the Johnson homomorphism $\tau_2$ on $\M1g[2]$  and maps to a multiple of $j_1$ under the natural projection.
\end{theorem}

In fact, much more is true:

\vskip .2in

\noindent Corollary~\ref{lifts} states that  \emph{all} the higher Johnson homomorphisms
lift to the groupoid level in a sense weaker than Theorem~\ref{t2exists} (see below and Section \ref{sect:morecochainreps});

\vskip .1in

\noindent Theorem~\ref{thm:tauisdeltaell} provides recursive formulae for the groupoid  Johnson lifts in terms of the (Campbell-)Hausdorff series (see Section~\ref{thegeneralcase});

\vskip .1in

\noindent Theorem~\ref{thm:tautensor} employs Poincar\'e duality on the surface to write the recursive formulae 
for the groupoid Johnson lifts as
explicit tensors in modules derived from the homology vector space;

\vskip .1in

\noindent  and Sections \ref{sect:explicitformulae}--\ref{sec63} 
 give progressively more complicated but explicit expressions for lifts of $\tau_1,\tau_2,\tau_3$, 
 (and to a lesser extent  $\tau _4$),
  and compare these results with those of Morita \cite{Morita89,Morita93a}.

\vskip .2in

\leftskip=0ex

It is worth emphasizing that it is only a matter of
patience to likewise derive explicit formulae on the groupoid level for $\tau_k$ with $k>4$, say, on the computer.
Using techniques of Morita, we further
massage our explicit formula for $\tau_2$ (see Section~\ref{sect:explicitformulae}) provided by  Theorem \ref{thm:tautensor}  to derive the 1-cocycle described in Theorem~\ref{t2exists}
(see Section~\ref{sect:cochainreps}).
We do not know if $\tau_k$ similarly lifts to a 1-cocycle for $k\geq 3$ (in no small part because
the appropriate coefficient modules are complicated) but nevertheless have
found explicit groupoid-level formulae in any case, which can be recursively computed
and are discussed in Section \ref{sect:morecochainreps}.

The proofs of these results rely on the existence of a canonical combinatorially defined ``fatgraph Magnus expansion'' $\theta^G:\pi_1\ra \widehat T$,
where $\widehat T$ is the ring of formal power series in several non-commuting variables,
and $\theta ^G$ is a suitable homomorphism.  Such a homomorphism $\theta ^G$
is associated to every  trivalent marked fatgraph $G$ and is described in terms of the induced 
 ``homology marking'' on its edges (see Section \ref{sect:Hmarkings}).   This is the heart of the paper (in Section \ref{sect:fatgraphmagnus}).
The derivation of this fatgraph Magnus expansion relies on certain combinatorial iterated integrals for such fatgraphs. 
Indeed, since one can construct  a different 
Magnus expansions as the holonomy of a certain flat connection  (cf. \cite{Kawazumi06}), it is natural that iterated integrals are involved.  Topologically, these integrals are taken along the path of the boundary $\partial \S1g$ beginning at the basepoint, while combinatorially they are iterated sums along the boundary edge-path of the fatgraph of specific
elements built from the homology marking.

Let us compare and contrast the approach of Morita-Penner with that of this paper (although Morita-Penner worked in  the context of punctured surfaces, their proofs work equally well in the bordered context).  In effect in \cite{moritapenner}, 
each ``Torelli space'', i.e., the quotient of Teichm\"uller space by a Torelli group, comes equipped with 
a natural ideal cell decomposition.  The combinatorial fundamental path groupoid of the dual 2-complex
gives a discretization of the fundamental path groupoid of Torelli space itself.  The oriented edges of this dual complex are in natural one-to-one correspondence with ``Whitehead moves'' (discussed in Section 3) between suitable  trivalent fatgraphs.  Furthermore, the 2-simplices of the dual complex give rise to three types of relations:
involutivity, commutativity, and the pentagon relation, which give a complete set of relations for the groupoid.   To  write a 1-cochain on the dual complex to Torelli space with values in some module, then, we must assign to each Whitehead move an element of the module.  This assignment is a 1-cocycle if and only if the involutivity, commutativity, and pentagon relations hold, and this is how it was confirmed in \cite{moritapenner} that the 1-cochain $j_1$ is actually a cocycle.

In this paper, we instead study generalized Magnus expansions, which in effect must satisfy two
key properties (orientation and vertex compatibility in Equations \eqref{eq:orientation} and \eqref{eq:vertex}), which are 
simpler than involutivity, commutativity, and the pentagon relation; in fact, we are able to
solve these two equations recursively here using the geometry of the underlying fatgraph.   

A basic
point of Kawazumi's theory \cite{Kawazumi05} is that the Johnson homomorphisms can be determined from any generalized Magnus expansion and in particular from our fatgraph Magnus expansion.  It follows that
our derived expressions for Johnson homomorphisms automatically satisfy
the involutivity, commutativity, and pentagon relations, a fact which is by no means obvious
from the formulae themselves.

The sense in which our groupoid lifts of $\tau_1,\tau_2$ are special is
that we have honest algebraic 1-cocycles with values in an appropriate module
in these two cases only.  In general for higher Johnson homomorphisms, we have only a
description as an invariant
of based homotopy classes of paths in Torelli space, hence a kind of ``cocycle with non-abelian coefficients''.  (See Section \ref{sect:morecochainreps} for further discussion.)

As a general point, we think it is no accident that the formalism here of homology marked fatgraphs
is similar to that of finite-type invariants of 3-manifolds, cf. \cite{garalevine98}, and our explicit calculations
here of higher Johnson homomorphisms determine (the tree-like part of) values of the latter under suitable circumstances \cite{garalevine05,Habegger,Levine,LevineAd}.  In fact, an appropriate setting for both discussions seems to be the theory of ``homology cobordisms'' \cite{chm}. 
 In this context, say in the setting of mapping tori of surface automorphisms, Whitehead moves can be seen to correspond to appropriate elementary moves on ideal triangulations
of the 3-manifold.  Thus, our formulae should provide purely combinatorial expression 
for (reductions of)  finite-type invariants of mapping tori under appropriate circumstances.  In this sense, 
Whitehead moves rather than Dehn twists indeed seem the natural generators.

One could say that ``here we give the first explicit calculations of the higher Johnson homomorphisms'',
but this is in a sense a swindle since, first of all, of course they can be computed from the definitions
in terms of the action of Dehn twists on the fundamental group, which is on the other hand completely unwieldy. The second swindle is that we have obviated here the need for describing generators of the higher
Torelli groups by working on the groupoid level, where all higher Torelli groupoids are generated by ``Whitehead moves on fatgraphs'' (see Section 3).  

On the other hand, this second swindle is not without significance, and we have truly
given a new closed form recursive expression for the higher Johnson homomorphisms on this groupoid level, at the very least, a new type of explicit algorithm for their calculation.   The most powerful
alternative method, which is presumably practicable for calculating $\tau _3$ and maybe $\tau _4$, is
to rely on the Magnus representation as described in \cite{Morita93a} and studied in \cite{Suzuki02,Suzuki05}.

We parenthetically mention that other homomorphisms on (subgroups of) mapping class groups can also be lifted to the groupoid level.  Namely, it is shown in \cite{abp}
that the ``Nielsen'' representation of the mapping class group of a (once-) bordered surface in
the automorphism group of a free group, the Magnus representation, and the symplectic representation all lift to the groupoid level, and explicit formulae are given.

This paper is organized as follows.  We begin with notation and background definitions of the Torelli groups, Johnson maps, and generalized Magnus expansions in Section 2.  In Section 3, we discuss the properties of fatgraphs and the fatgraph complex which will be needed for our results.  Section 4 introduces the fatgraph Magnus expansion and the fatgraph Johnson maps.   In Section 5, we reinterpret the results of the previous section in terms of invariant tensors with explicit expressions
for $\tau_1,\tau_2,\tau_3$ in Section \ref{sec51} and explicit cocycle representatives for the former two
in Section \ref{sect:cochainreps}.  In Section \ref{sect:examples}, we illustrate our techniques with examples of  twists on separating curves calculating $\tau_1,\tau_2$ in Section \ref{sec61}, providing a general procedure for computing $\tau _3$ in Section \ref{sec62}, and giving  explicit values  of $\tau_3$ for certain  separating twists related to Lickorish's  generators in Section \ref{sec63}.

\section{Torelli groups and Johnson homomorphisms}

We begin by establishing notation.  Let $\S1g$ denote a fixed oriented surface of genus $g>0$ with one boundary component, and let $p$ and $q$ be distinct points on the boundary. (The point $q$ will not be needed until the next section.)  Let $\pi_1=\pi_1(\S1g,p)$ denote the fundamental group of this surface relative to the basepoint $p$, which is isomorphic to a free group $F_{2g}$ on $2g$ generators.  Let $\partial\in\pi_1$ denote the element represented by the path of the boundary $\partial\S1g$, which we sometimes consider as a word in the generators of $F_{2g}$.   For example, in terms of the standard symplectic generators of $\pi_1$ shown in Figure \ref{fig:symplecticbasis}, we have $\partial=\prod_{i=1}^g[\mb{u}_i,\mb{v}_i]$.

\begin{figure}[!h]
\begin{center}
\epsffile{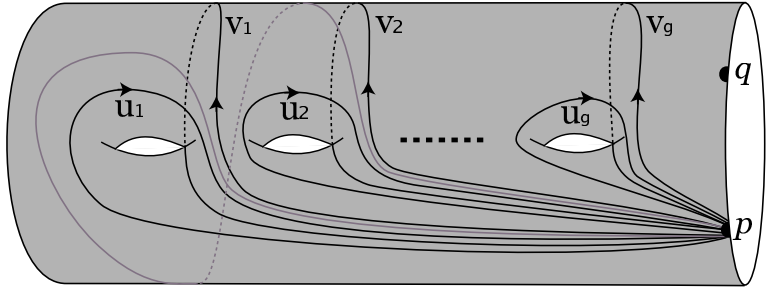}
\caption{A symplectic basis for $\pi_1$.}
\label{fig:symplecticbasis}
\end{center}
\end{figure}

 We define the {\it mapping class group}  of $\S1g$ to be the  isotopy classes of self-diffeomorphisms of $\S1g$ which fix the boundary pointwise, $$\M1g=\pi_0(\Diff(\S1g),\partial\S1g),$$ where isotopies are required to also fix the boundary pointwise.   $\M1g$ acts on $\pi_1$ in the obvious way, and by a classical result of Dehn-Nielsen, we have
 \[
 \M1g\cong\{  f\in\Aut(F_{2g}) \; | \; f(\mb{\partial})=\mb{\partial}   \}.
 \]
$\M1g$  acts also on  the abelianization of $\pi_1$ which we denote by $H\cong H_1(\S1g,\bZ)$. The  kernel of this  action $\M1g \ra \Aut(H)$ is nothing but the classical Torelli group $\I1g$, while the image is classically known to be $Sp(2g,\bZ)$.  

More generally, let $N_k$ denote the $k$th term of the lower central series of $\pi_1$ so that $N_k=\pi_1/\Gamma_k$ with $\Gamma_0=\pi_1$ and $\Gamma_{k+1}=[\Gamma_{k},\Gamma_0]$.  Thus, 
$\M1g$ acts on each $N_k$,  and we define the $k$th {\it Torelli group} to be the kernel of the map 
\[
\M1g[k]=\Ker (\M1g\ra \Aut(N_k)  ).
\]

Using the short exact sequence
\[
0\ra \Gamma_k/\Gamma_{k+1}\ra N_{k+1} \ra N_k \ra 1,
\]
one can show that for each element $\varphi\in\M1g[k]$ and each element $x\in N_{k+1}$ we have  $\varphi(x)x^{-1}\in\Gamma_k/\Gamma_{k+1}$.  By a classical result of Magnus, it is known that $\Gamma_k/\Gamma_{k+1}\cong \El_{k+1}$ where $\El_{k+1}$ is the $(k+1)$st graded component of the free Lie algebra  on $H$ \cite{mks}.  It can be shown that this leads to a homomorphism
\[
\tau_k: \M1g[k]\ra \Hom(H,\El_{k+1})\cong H\otimes \El_{k+1},
\]
where we have made implicit use of the \Poin duality of the surface $H^*\cong H$
i.e., $\Hom(H,\El_{k+1})\cong H^*\otimes \El_{k+1}\cong
H\otimes\El _{k+1}$.
 These  maps $\tau_k$ are the $k$th {\it Johnson homomorphisms}.  By definition, the kernel of $\tau_k$ is precisely $\M1g[k+1]$, while 
Morita \cite{Morita93a} has shown  that its image under $\tau _k$ is a  submodule of $\mc{H}_k\subset H\otimes \El_{k+1}$, where $\mc{H}_k$ is the kernel of the bracket map $H\otimes \El_{k+1}\ra\El_{k+2}$.

\subsection{Generalized Magnus expansions}
Let $F$ be a free group, whose abelianization was denoted $H$ in the previous section.
From here on, we shall work over the rationals and henceforth let $H$ denote the tensor
product of the abelianization of $F$ with the rationals ${\mathbb Q}$, i.e., $H=H_1(\S1g;{\mathbb Q})$.
Let
$\widehat T$ denote the completed tensor algebra of $H$, 
$$\widehat T = \prod_{i=0}^\infty H^{\otimes i},$$
so $\widehat T$ is naturally identified with the ring of all formal power series
in generators of $H$.
Note that $\widehat T$  is naturally  filtered by ideals  $\widehat T_p = \prod_{i=p}^\infty H^{\otimes i}$.   

A {\it generalized Magnus expansion} of $F$ (over $\mathbb Q$) in the sense of \cite{Kawazumi05}  is a group homomorphism $$\theta: F~~\ra ~~1+ \widehat T_1$$ such that $$\theta(\mb{a})=1+a +\theta_2(\mb{a}) +\theta_3(\mb{a}) +\ldots$$ for any ${\bold a} \in F$, 
where $\theta_i({\bold a}) \in 
H^{\otimes i}$ is the $i$th graded component of the tensor $\theta({\bold a}) \in \widehat{T}$ and $a=[{\bold a }]\in H$.  
The {\it standard} Magnus expansion (studied by Magnus {\it et al.}) is the simplest possible
(non-canonically however) :
$\theta ({\bold x}_i)=1+x_i$, for some prescribed generating set $\{{\bold x}_i\}$ of $F$.  The group $IA(\widehat T)\subset \Aut(\widehat T)$ of filter-preserving algebra isomorphisms acting trivially on $\widehat T_1/\widehat T_2$  acts transitively and freely on the space of all Magnus expansions and can be identified with $\Hom(H,\widehat T_2)$ by $U\mapsto U|_H - id|_H$ for $U\in IA(\widehat T)$, cf.  \cite{Kawazumi05}.   

Now consider the special case when $F=\pi_1$ and let  $\varphi\in\I1g$ and $\theta:\pi_1\ra \widehat T$ be any Magnus expansion.  
 $\theta$ induces a natural isomorphism $\widehat{\mathbb Q[\pi_1]}\cong \widehat{T}$, with $\widehat{\mathbb Q[\pi_1]}$ the completed group algebra of $\pi_1$ with respect to the augmentation ideal, and the action of $\varphi$ on $\pi_1$ extends to an automorphism of $\widehat{\mathbb Q[\pi_1]}$. Following  \cite{Kawazumi05}, we define the ``total Johnson map'' $\tau^\theta(\varphi)$ as 
\[
\tau^\theta(\varphi) = \theta\circ\varphi\circ\theta^{-1}: 
\widehat{T} \to \widehat{\mathbb Q[\pi_1]} \to \widehat{\mathbb Q[\pi_1]}
\to \widehat{T}.
\]
Let $\tau^\theta_m(\varphi)$ denote the $m$th component of $\tau^\theta(\varphi)$ considered as an element in $\Hom(H ,H^{\otimes m+1})$ by the association $IA(\widehat T)\cong \Hom(H,\widehat T_2)$.  

Kitano \cite{Kitano} first expressed the Johnson homomorphisms in terms of the standard
Magnus expansion.  Kawazumi's generalization \cite{Kawazumi05} introduced generalized Magnus expansions and extended the domain of the Johnson homomorphisms (as maps) to mapping class groups and  beyond to automorphism groups of free groups.

In particular for $m>0$, $\tau^\theta_m$ gives a group homomorphism $\tau^\theta_m : \M1g[m]\ra H^*\otimes H^{\otimes (m+1)}$ which is independent of the Magnus expansion $\theta$ and coincides with the Johnson homomorphism.  From now on, we shall drop the dependence of $\theta$ in the notation.
For $\varphi\in\M1g[m]$, we have
\[
\tau_m(\varphi)(x)=\theta_m(\varphi (\mb{x}))-\theta_m(\mb{x}),
\]
 while $\tau_j(\varphi)(x)$ vanishes when $j<m$.

In particular, given an element $\varphi\in\I1g={\mathcal M}_{g,1}[1]$, the classical Johnson homomorphism $\tau_1$ is described by 
\[
\tau_1(\varphi)(x)=\theta_2(\varphi(\mb{x}))-\theta_2(\mb{x}).
\]
The map  $\tau_2$, however,   is not a group homomorphism on $\I1g$, but rather is described by 
\[
\tau_2(\varphi)(x)=\theta_3(\varphi(\mb{x})) -\theta_3(\mb{x}) -(\tau_1(\varphi)\otimes 1+1\otimes\tau_1(\varphi) ) \circ \theta_2(\mb{x}),
\]
which restricts to  a group homomorphism on $\mc{K}_{g,1}={\mathcal M}_{g,1}[2]$ since $\tau_1$ vanishes there.

\vskip .2in

\section{The fatgraph complex and Torelli groupoids}

Recall \cite{Penner88} that a {\it fatgraph} is a connected 1-dimensional CW complex with a prescribed cyclic ordering  of all  half edges incident at each vertex.  The cyclic ordering at each vertex can be used to define cycles of oriented edges, called \emph{boundary cycles}, by associating to an incoming edge an outgoing edge which is next in the cyclic ordering.  We say that a fatgraph has genus $g$ if gluing a 2-cell along each boundary cycle produces a closed surface of genus $g$.

We say a fatgraph is a \emph{(once-)bordered fatgraph} if there is only one boundary cycle, and all vertices are of valence at least three except for a single univalent vertex.  We call the edge incident to this univalent vertex the {\it tail}.  We shall usually give the tail an orientation so that it points away from the univalent vertex and denote this oriented edge by $\mb{t}$.  

A \emph{marking} of a genus $g$ bordered fatgraph $G$ is a homotopy class of embeddings $f:G\hra\S1g$ such that the complement $\S1g\backslash f(G)$ is contractible and $G\cap\partial\S1g=t\cap\partial \S1g = \{q\}$.  Note that the mapping class group acts naturally on the set of markings of $G$, 
hence so too do the Torelli groups.  \Poin dual to any marked bordered fatgraph is a filling arc family (i.e., every essential closed curve in $\S1g$ meets the arc family), and moreover, this arc family can be chosen so that all arcs are based at  the basepoint $p$.  This is illustrated in the Figures \ref{fig:symplecticbasis} and \ref{fig:4ggon}, where we have chosen the dual arc family to be an extension of the standard set of symplectic generators of $\pi_1$. 
  
\begin{figure}[!h]
\begin{center}
\epsffile{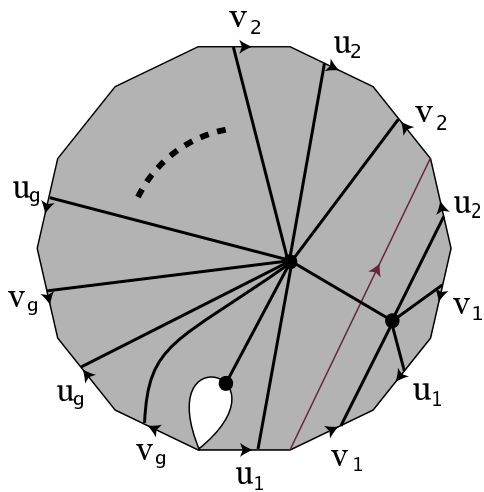}
\caption{Fatgraph and dual arc family extending the standard symplectic basis of $\S1g$.}
\label{fig:4ggon}
\end{center}
\end{figure}
  
Using the orientation of the surface, this allows for an association of an element of $\pi_1$ to each oriented edge of a marked bordered fatgraph, where $\partial$ is chosen to be associated to the tail with its reversed orientation $\mb{\bar t}$.  

 Let $\mc{E}_{or}(G)$ denote the set of oriented edges of $G$ and $\bar{\mb{e}}$ denote the edge $\mb{e}\in \mc{E}_{or}(G)$ with the opposite orientation.
 
\begin{definition}\label{defnofmarking}
A $\pi$-marking of a bordered fatgraph $G$ is a map $\pi:\mc{E}_{or}(G)\ra\pi_1$ 
which satisfies the following conditions:
\begin{itemize}
\item{\bf (orientation)} For every oriented edge $\mb{e}$ of $G$, 
\[
\pi(\bar{\mb{e}})\pi(\mb{e})= 1;
\]
\item {\bf (vertex compatibility)} For every vertex $v$ of $G$,
\[
\pi(\mb{e}_1)\pi(\mb{e}_2)\dotsm\pi(\mb{e}_k)=1,
\]
where  $\mb{e_i}$ are the cyclically ordered edges incident of $v$ and oriented pointing inwards $v$; 
\item {\bf (full rank)}
$\im(\pi)$ generates $\pi_1$;
\item {\bf (geometricity)}
$\pi(\mb{\bar t})=  \partial$.
\end{itemize}
\end{definition}

 More generally, define an {\it abstract $K$-marking} of $G$ for any group $K$ as a map which satisfies the corresponding orientation and vertex compatibility conditions.  We say that an abstract $N_k$-marking of $G$ is \emph{geometric} if it descends from a $\pi$-marking under the quotient map $\pi_1\ra N_k$.  

\begin{lemma}\label{hopfian}
The two notions of markings are equivalent, i.e., there is a natural bijection between the set of marked bordered fatgraphs and the set of $\pi$-marked bordered fatgraphs.
\end{lemma}
\begin{proof}
We have already shown that each marking $f:G\hra\S1g$ of a bordered fatgraph $G$ leads to a map from $\mc{E}_{or}(G)$ to $\pi_1$, and it is immediately verified that this  map is a $\pi$-marking of $G$.   We need to exhibit the inverse map.  Let $(G,\pi)$ be a $\pi$-marked fatgraph and let  $f:G\hra\S1g$  be any marking of $G$.  Let $X$ be a set of oriented (distinct) edges whose complement in $G$ is a maximal tree.   It is easy to see that the  group elements associated to the set  $f(X)$ give a set of generators of $\pi_1$. Thus, there is  a homomorphism of $\pi_1$ onto itself which  takes $f(X)$ to $\pi(X)$, and  by the Hopfian property 
\cite{mks} of $\pi_1$, this is an isomorphism.  Since 
 this isomorphism must take $f(\mb{t})$ to $\pi(\mb{t})$, it lies in $\M1g\subset\Aut(\pi_1)$.   If $\phi$ is a diffeomorphism representing this mapping class, then $\phi\circ f:G\hra\S1g$ gives an induced marking of $G$.  It is easy to see that this is the inverse of the original map.
\end{proof}

From now on we shall assume all fatgraphs are once-bordered fatgraphs.  We shall also often assume that a fatgraph $G$ comes with a particular marking, although we may at times suppress the marking  from the notation.  In particular, for $\varphi\in\M1g$ we shall denote by $\varphi(G)$ the fatgraph $G$ with marking altered by post-composing with $\varphi$.

Given any non-tail edge  $\mb{e}$ of $G$ with distinct endpoints, we can collapse $e$ to obtain a new fatgraph $G'$.  Moreover, using the vertex compatibility condition, any $\pi$-marking of $G$ determines a unique $\pi$-marking of $G'$ and vice versa.  

In particular, consider the elementary \emph{Whitehead move} which collapses an edge $e$ (with distinct endpoints) of a trivalent fatgraph $G$ and then un-collapses this vertex in the unique distinct way producing a new edge $f$ and a new fatgraph $G'$.   We shall often denote such a Whitehead move
along $e$ simply as $W:G\to G'$.  The evolution of the $\pi$-markings under such a move is  summarized in Figure \ref{fig:caseI}. 

Recall \cite{Harer86,Penner87,Penner04} that 
the {\it fatgraph complex} $\mc{G}_T$ by definition has an open $k$-simplex for every marked bordered fatgraph with $k+1$ non-tail edges and has face relations  given by the evolution of the markings under edge collapse.  It is known \cite{Harer86,Penner87,Penner04} that the geometric realization $|\mc{G}_T|$ is homeomorphic to an open ball (namely, a decorated Teichm\"uller space), thus \Poin dual to this complex is a homotopically equivalent complex we denote by $\hat{\mc{G}}_T$.  $0$-cells of $\hat{\mc{G}}_T$ correspond to trivalent marked fatgraphs and  oriented $1$-cells correspond to Whitehead moves between such graphs.  The $2$-cells of $\hat{\mc{G}}_T$ come in two varieties, corresponding  to marked fatgraphs which either have two $4$-valent vertices or one $5$-valent one with the rest trivalent.  

 $\mc{G}_T$, and thus also  $\hat{\mc{G}}_T$,  is a contractible complex upon which $\M1g$ acts freely and properly discontinuously.  Let  $\mc{G}_M$ be the quotient complex under this action, so that $\M1g\cong\pi_1(|\mc{G}_M|)$.  Similarly,  let  $\mc{G}_I$ be the quotient under the action by $\I1g$ so that $\I1g\cong\pi_1(|\mc{G}_I|)$ .  Similarly define $\hat{\mc{G}}_M$ and $\hat{\mc{G}}_I$.  Note that  $0$-cells in $\hat{\mc{G}}_M$ correspond to equivalence classes of (unmarked) trivalent bordered fatgraphs $G$.  

Since each element of ${\mathcal M}_{g,1}$ can be represented by a loop in $|\hat{\mc{G}}_M|$, it can dually be represented by a path in $|\hat{\mc{G}}_T|$ corresponding to a sequence of Whitehead moves on trivalent marked fatgraphs beginning and ending on isomorphic (unmarked) fatgraphs.
Furthermore, this representation is unique modulo  the commutativity and pentagon relations corresponding to the non-degenerate 2-cells in $\hat{\mc{G}}_T$ and the involutivity relation
corresponding to degenerate 2-cells.

Thus, the mapping class group ${\mathcal M}_{g,1}$ is realized as the stablizer of any point in the fundamental path groupoid of the manifold $|\mc{G}_M|$.  Likewise following \cite{moritapenner}, we define the $k$th {\it Torelli groupoid} to be the fundamental path groupoid of the manifold
$|\mc{G}_T|/{\mathcal M}_{g,1}[k]$, and in particular, the {\it (classical) Torelli groupoid} is the fundamental path groupoid of
$|\mc{G}_I|=|\mc{G}_T|/{\mathcal M}_{g,1}[1]$.

\subsection{Homology markings}
\label{sect:Hmarkings}

We now wish to describe the cells in $\hat{\mc{G}}_I$, each of which  is an $\I1g$-orbit of a marked fatgraphs.  These can be identified with geometrically $H$-marked fatgraphs, meaning the $H$-markings descend from a $\pi$-marking under the abelianization map $\pi_1\ra H$.  Thus, we are lead to consider the problem of recognizing which $H$-markings of a fatgraph are geometric.

 By abuse of notation, we shall usually denote the $\pi$-marking of an oriented edge $\mb{e}$ simply by boldface $\mb{e}=\pi(\mb{e})$ while we will denote the corresponding $H$-marking simply by the lowercase $e=H(\mb{e})=[\mb{e}]$.   We  will also  denote the homology intersection pairing of $H$ provided by the \Poin duality of $\S1g$ by a dot $\; \cdot\;:H\times H \ra \bQ$.    

Given a fatgraph $G$, we define a skew-symmetric pairing $\langle\;\, ,\;\rangle:\mc{E}_{or}(G)\times\mc{E}_{or}(G)\ra\{1,-1,0\}$
as follows.  Given two oriented distinct  edges $\mb{a}$ and $\mb{b}$, consider the path along the boundary cycle of $G$ starting at $\mb{a}$.  If the sequence traverses the edges in order $\mb{a}$, $\mb{b}$, $\mb{\bar a}$, $\mb{\bar b}$, then we set $\langle \mb{a},\mb{b}\rangle=-1$.  If the order is  $\mb{a}$,  $\mb{\bar b}$, $\mb{\bar a}$, $\mb{b}$, then we set $\langle\mb{a},\mb{b}\rangle=1$.  Otherwise, we set $\langle\mb{a},\mb{b}\rangle=0$.  We also set $\langle\mb{e},\mb{e}\rangle=\langle\mb{e},\mb{\bar e}\rangle=0$ for all oriented edges $\mb{e}$.

One can directly check  that for any geometric $H$-marking  of $G$, the homology intersection pairing  
 matches (under our orientation conventions) the skew pairing on oriented edges defined above, meaning $\langle\mb{a},\mb{b}\rangle=a\cdot b$ for all oriented edges $\mb{a}$ and $\mb{b}$ of $G$.  
  In fact, we have also the reverse implication:
\begin{proposition}\label{geomh}
An $H$-marking on a once-bordered fatgraph $G$ is geometric if and only if
$\langle\mb{a},\mb{b}\rangle=a\cdot b$, for all oriented edges $\mb{a},\mb{b}$ of $G$.
\end{proposition}
\begin{proof}
The proof follows that of Lemma~\ref{hopfian} based on surjectivity of the map $\M1g\ra Sp(2g,\bZ)$.
\end{proof}

\section{Fatgraph Magnus expansions and Johnson maps}
\label{sect:fatgraphmagnus}
 
\begin{figure}[!h]
\begin{center}
\epsffile{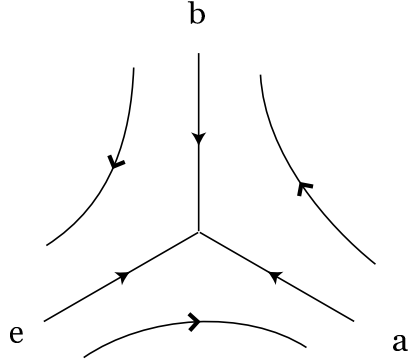}
\caption{Labeling of the vertex}
\label{fig:vertex}
\end{center}
\end{figure}

We now wish to describe a canonical Magnus expansion $\theta^G:\pi_1\ra \widehat T$ for every trivalent marked bordered fatgraph $G$.  The goal is to construct $\theta^G$ as a $(1+\widehat T_1)$-marking of $G$ such that  for any oriented edge $\mb{e}\in \mathcal{E}_{or}(G)$,  $\theta^G(\mb{e})=1+e+\theta^G_2(\mb{e})+ \ldots$ is described purely in terms of the combinatorics and the induced $H$-marking of $G$.   

 For this, consider a vertex of $G$ labeled as in Figure \ref{fig:vertex}. 
Any solution $\theta^G$ must satisfy the 
orientation  
\begin{equation}\label{eq:orientation}
\theta^G(\mb{e})\theta^G(\mb{\bar e})=1
\end{equation}
and
vertex compatibility condition
\begin{equation}\label{eq:vertex}
\theta^G(\mb{a})\theta^G(\mb{b})\theta^G(\mb{e})=1.
\end{equation}

As a warm-up to the general solution, we begin by solving for $\theta^G$ up to degree 2.  We let $\exp(e)$ serve as the first approximation to $\theta^G(\mb{e})$ and write 
\[
\theta^G(\mb{e})=1+e+\half e^2 +\ell_2(\mb{e}) +\theta^G_3(\mb{e}) +\ldots
\] 
and similarly for $\theta^G(\mb{a})$ and $\theta^G(\mb{b})$,
where $e^2$ denotes $e\otimes e$ and $\ell_2(\mb{e})$ captures the noncommutativity of \eqref{eq:vertex} in degree  2.  Note that the orientation condition forces $\ell_2(\mb{e})=-\ell_2(\mb{\bar e})$, which motivates setting 
\[
\ell_2(\mb{e})=g_2(\mb{e})-g_2(\mb{\bar e}),
\]
 with $g_2$ some function on oriented edges of $G$ for which we must solve.

Expanding \eqref{eq:vertex}  to degree 2 and using the fact that $a+b+e=0$, we find
\[
\begin{split}
-\ell_2(\mb{e})-\ell_2(\mb{a})-\ell_2(\mb{b})&=a\otimes b+a\otimes e+b\otimes e  +\half(a^2+b^2+e^2)  \\
&=\half ( a\otimes b -b \otimes a) \\
&= \sixth([a,b] +[b,e]+[e,a]),
\end{split}
\]
where we have denoted $x \otimes y - y \otimes x$ by $[x,y]$.

 Rewriting in terms of $g_2$, we have
\[
g_2(\mb{\bar a})+g_2(\mb{\bar b})+g_2(\mb{\bar e})-g_2(\mb{ a}) -g_2(\mb{b})-g_2(\mb{e})=\sixth([a,b] +[b,e]+[e,a]),
\] 
which can be  solved by setting 
\[
g_2(\mb{\bar b})-g_2(\mb{a})= \sixth [a,b]
\]
and similarly for the other cyclic permutations of $a,b,e$.  
The above expression gives a difference formula for $g_2$ (with respect to the boundary cycle), which can be solved by ``integrating'' 
\[
g_2(\mb{e})=\sixth \sum_1^k [f_{i-1},-f_i],
\]
where the sum is over edges in the boundary cycle of $G$ starting at the tail and ending at $\mb{e}$ so that $\f_0=\mb{t}$,  $\f_k=\mb{e}$, and all $\f_i$ are oriented with the direction of the boundary cycle. In this way, for any oriented edge $\mb{x}$ of $G$ we have 
\[
\ell_2(\mb{x})=\sixth \sum_1^k [f_{i-1},f_i],
\]
where the sum is now over the edge-path $\{\f_i\}$ of the boundary cycle connecting $\mb{x}$ to $\mb{\bar x}$ which avoids the tail $\mb{t}$ if such a path exists.  See  Figure \ref{line2}. If such a path does not exist for $\mb x$, then it does for $\bar {\mb x}$, in which case   $\ell_2(\mb{x})=-\ell_2(\mb{\bar x})$. 

\begin{figure}[!h]
\begin{center}
\epsffile{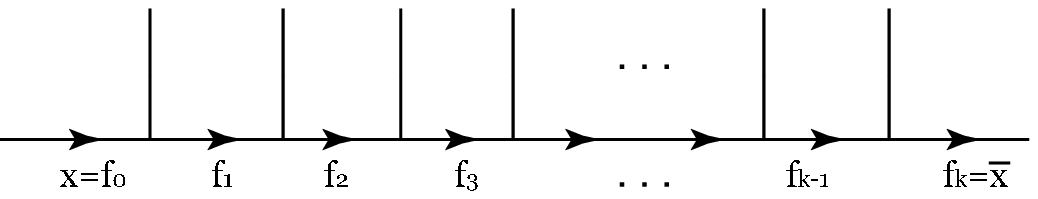}
\caption{Edges from $\mb{x}$ to $\mb{\bar{x}}$}
\label{line2}
\end{center}
\end{figure}

\subsection{The general case}\label{thegeneralcase}
Now we return to the general case armed with the insights of our warm-up $k=2$ and determine a method for recursively defining $\theta_k^G$ for $k>2$.  For this, we again take $\exp(e)$ as a first approximation but now make use of the Hausdorff series \cite{bourbaki} 
for the free Lie algebra $\El$ on $H$ to capture the noncommutativity of \eqref{eq:vertex}.   
If we let $X+Y$ + $h(X,Y)$ denote this series for $X,Y\in\El$, meaning 
\[
X+Y + h(X,Y) = \log (\exp X \exp Y),
\]
then the first few terms are given by 
\begin{equation} \label{eq:haus}
\begin{split}
 h(X, Y)&=\frac12 [X,Y] +\frac{1}{12}[X,[X,Y]] - \frac{1}{12}[Y,[X,Y]]\\
 &-\frac{1}{24}[X,[Y,[X,Y]]] +\dotsm.
 \end{split}
\end{equation}

Now, we shall assume  there exists some $\ell(\mb{x})\in \El$ for each oriented edge $\mb{x}$ of $G$ such that our desired Magnus expansion takes the form 
\[
\theta^G(\mb{x})=\exp(\ell(\mb{x})).
\]
We let $\ell_n(\mb{x})$ denote the $n$th graded component of $\ell(\mb{x})$  and observe that  by defiinition, we must have $\ell_0(\mb{x})=0$ and $\ell_1(\mb{x})=x$.  By the results of the previous section, we may choose  $ \ell_2(\mb{x})$ to be defined as  a combinatorial integral over an edge-path connecting $\mb{x}$ to $\bar{\mb{x}}$ (or its reverse); similarly, we shall see that for all $m\geq 2$, we may choose $ \ell_m(\mb{x})$ to be an iterated combinatorial integral over the same edge path.

From the relation  $\theta^G(\mb{x})\theta^G(\overline{\mb{x}}) = 1$, we must have  $\ell(\overline{\mb{x}}) = - \ell(\mb{x})$, and we again assume $\ell_n(\mb{x}) = g_n(\mb{x}) - g_n(\bar{\mb{x}})$ for some  unknown function $g_n$ when $n > 2$.  From $\theta^G(\overline{\mb{e}}) = \theta^G(\mb{a})\theta^G(\mb{b})$, it  follows that
\begin{equation}\label{eq:hvtxcompatibility}
- \ell(\mb{e}) = \ell(\overline{\mb{e}}) = \ell(\mb{a}) + \ell(\mb{b}) + h(\ell(\mb{a}), \ell(\mb{b})).
\end{equation}
Symmetrizing this relation, we have
\begin{multline}\label{eq:symvtxcompatibility}
-3 (\ell(\mb{a}) + \ell(\mb{b}) + \ell(\mb{e}))=\\
 h(\ell(\mb{b}), \ell(\mb{e})) + h(\ell(\mb{e}), \ell(\mb{a})) + h(\ell(\mb{a}), \ell(\mb{b})),
\end{multline}
hence we choose
\[
3(g_n(\bar{\mb{a}}) - g_n(\mb{e})) = h(\ell(\mb{e}), \ell(\mb{a}))_{(n)}
= h(\ell(\mb{e}), -\ell(\bar{\mb{a}}))_{(n)},
\]
where $_{(n)}$ denotes the $n$-th graded component in $\mathcal{L}$. 

This gives a difference equation for $g_n$, which we can integrate along any edge-path of $G$.  In particular, we have the following:

\begin{definition}\label{defn:ell}
Given a trivalent marked bordered fatgraph $G$, we recursively define a map $\ell^G:\mc{E}_{or}(G)\ra\El$ by setting 
\begin{equation}\label{eq:elln}
\ell^G_n(\mb{x}) = -\frac13 \sum^k_{i=1} h(\ell^G(\f_{i-1}), -\ell^G(\f_i))_{(n)}
\end{equation}
for any oriented edge $\mb{x}$ such that the edge-path $\{\f_i\}$  from $\mb{x}$ to $\mb{\bar x}$ along the boundary cycle avoids the tail.  
If no such path exists for $\mb{x}$, we set  $\ell^G_n(\mb{x})=-\ell^G_n(\mb{\bar x})$.  

Note that  this recursive definition is well-posed since the right hand side of the expression for  $\ell^G_n\in\El_n$ involves only the terms  $\ell^G_k$ with $k<n$.  
\end{definition}

Using this definition for $\ell(\mb{x})$, we have:
\begin{theorem} \label{thm:fatgraphmagnus}
For any  trivalent  marked bordered fatgraph $G$, the map $\theta^G:\mc{E}_{or}(G)\ra\widehat{T}$ given by $\theta^G: \mb{x}\mapsto \exp(\ell^G(\mb{x}))$ extends to a generalized  Magnus expansion such that for $\mb{x}\in\mc{E}_{or}(G)$, $\theta^G(\mb{x})$ depends only on the combinatorics and $H$-marking of $G$.  Moreover, this \emph{fatgraph Magnus expansion}  is canonically defined and $\M1g$-equivariant in the sense that 
\[
\theta^G(\mb{x})=|\varphi^{-1}|~ \theta^{\varphi(G)}(\varphi(\mb{x}))
\]
for all $\varphi\in \M1g$ and  $\mb{x}\in\pi_1$, where $|\varphi^{-1}|$ is the element of $Sp(2g,\bZ)$ induced by $\varphi^{-1}$. 
\end{theorem}
\begin{proof}
By construction, the mapping $\ell^G_n:\mc{E}_{or}(G)\ra\El$ a priori only satisfies the symmetrized vertex compatibility relation \eqref{eq:symvtxcompatibility} but not necessarily \eqref{eq:hvtxcompatibility}. 
We must check that $\ell^G_n$ is inherently cyclicly  symmetric, meaning that 
\[
h(\ell^G(\mb{e}), \ell^G(\mb{a}))_{(n)} = h(\ell^G(\mb{a}), \ell^G(\mb{b}))_{(n)}.
\]
We prove this by induction and begin with the hypothesis that 
\[
\theta ^G(\mb{e})\theta ^G(\mb{a})\theta ^G(\mb{b}) \equiv 1 \textrm{ mod } \widehat{T}_n.
\]

For convenience, we simply write  $X*Y$ for $\log(\exp X \exp Y)$ if $X,Y \in \mathcal{L}$.  The associative law of the group $\exp(\mathcal{L})$ implies
\[
h(X*Y, Z) + h(X, Y) = h(X, Y*Z) + h(Y, Z)
\]
for any $X, Y$ and $Z \in \mathcal{L}$.

Now, we have
\[
h(\ell^G(\mb{e})*\ell^G(\mb{a}), \ell^G(\mb{b}))_{(n)} = h(-\ell^G(\mb{b}), \ell^G(\mb{b}))_{(n)} = 0
\]
since
\[
-\ell^G(\mb{b}) = \ell^G(\overline{\mb{b}}) = \ell^G(\mb{e})*\ell^G(\mb{a}) \textrm{ mod } \widehat{T}_n
\]
and $h(X, Y)_{(n)}$ is determined by the terms of degree $< n$.
Similarly, $h(\ell^G(\mb{e}), \ell^G(\mb{a})*\ell^G(\mb{b}))_{(n)} = 0$. and hence
\[
h(\ell^G(\mb{e}), \ell^G(\mb{a}))_{(n)} = h(\ell^G(\mb{a}), \ell^G(\mb{b}))_{(n)},
\]
as was to be shown.  

The above argument shows that   $\ell^G$ satisfies the (logarithmic) vertex compatibility condition \eqref{eq:hvtxcompatibility} for all vertices of $G$.  Thus, $\ell^G$ as defined can be extended to a map $\ell^G:\pi_1\ra \El$, and it is clear by construction that $\theta^G=\exp(\ell^G)$ gives a well-defined Magnus expansion satisfying \eqref{eq:orientation} and \eqref{eq:vertex}. 

The remaining statements in the theorem follow by construction since Definition \ref{defn:ell}  requires only the combinatorics and $H$-marking of $G$ and involves no choice of basis for $\pi_1$. 
\end{proof}

We emphasize that the above proof only holds for trivalent bordered fatgraphs.  A similar result can likely be proven for arbitrary bordered fatgraphs by finding symmetrized solutions for higher valence vertex compatibility conditions; however, it seems difficult to find a truly canonical or universal fatgraph Magnus expansion for all bordered fatgraphs which varies continuously under edge collapse.

\subsection{Johnson maps}

As a consequence of Theorem \ref{thm:fatgraphmagnus}, we find  

\begin{corollary}\label{white}
Let 
\[
 G=G_0\xra{W_1} G_1\xra{W_2} G_2 \xra{W_3} \ldots \xra{W_k}G_k =\varphi (G)
\]
be a sequence of (marked) Whitehead moves representing an element $\varphi\in\M1g[m]$ with $m\geq1$. Then
\[
\tau_{m}(\varphi)(x)=\sum_{i=1}^k  \theta_{m+1}^{G_{i-1}}(\mb{x})  -\theta_{m+1}^{G_i}(\mb{x}) 
\]
for any $\mb{x}\in \pi_1$.
\end{corollary}
\begin{proof}
For any $\mb{x}\in \pi_1$, we have 
\[
\begin{split}
\tau_{m}(\varphi)(x)&=\theta_{m+1}^{G}(\varphi (\mb{x})) -\ThG{m+1}{G}{x}\\ 
&=\theta_{m+1}^{\varphi (G)}(\varphi(\mb{x})) -\ThG{m+1}{\varphi (G)}{x}\\ 
&=\ThG{m+1}{G}{ x} -\ThG{m+1}{\varphi (G)}{x}\\
&=\sum_{i=1}^k \bigl [ \theta_{m+1}^{G_{i-1}}(\mb{x})  -\theta_{m+1}^{G_i}(\mb{x}) \bigr ],
\end{split}
\]
where the first line follows from the definition,  the second line follows since  $\tau_m$ is independent of the Magnus expansion, the third line follows from the $\M1g$-equivariance of $\theta^G$ given by Theorem \ref{thm:fatgraphmagnus}, and the last line follows since the sum telescopes.  
\end{proof}

Thus, by the $\M1g$-equivariance of the fatgraph Magnus expansion, we  have the following  immediate consequence:

\begin{corollary}\label{lifts}
The $m$th Johnson homomorphism $\tau _m$  lifts to the Torelli groupoid, for $m\geq 1$.
\end{corollary}

While Corollary \ref{white} certainly proves such a lift exists, the resulting expression is in many ways unsatisfying.   The bulk of the rest of this paper is an explication of these
lifts and a massaging of the formulae to more satisfactory expressions for $\tau _m$ in general and for $m=1,2,3$ in particular.

 One respect in which the lifts  of Corollary \ref{white} are lacking is that the summands  $\theta^{G_{i-1}}_{m+1}(\mb{x})-\theta^{G_{i}}_{m+1}(\mb{x})$ are not necessarily linear in $x$.  This motivates the following definition.  
 
\begin{definition} 
Given a Whitehead move $W:G\ra G'$, we define the \emph{fatgraph Johnson map} $\tau^{GG'}=\tau(W)$ to be the element of $IA(\widehat T)\cong \Hom(H,\widehat T_2)$ taking $\theta^{G'}$ to $\theta^{G}$:
 \[
\tau^{GG'}=\theta^{G}\circ(\theta^{G'})^{-1}.
 \]
\end{definition} 

With this definition, Corollary~\ref{white}, for $m=1,2$, can be put into the following more desirable form:
\begin{theorem}
In the notation of Corollary~\ref{white}, we have 
 $$\tau _1(\varphi)=\sum_{i=1}^k \tau_1^{G_{i-1}G_i},~for~\varphi\in{\mathcal M}_{g,1},$$ and
\begin{equation}\label{eq:localtau2sum}
\tau_2(\varphi)=\sum_{i=1}^k \tau_2(W_i)+( \tau_1(W_i)\otimes1+1\otimes\tau_1(W_i))\circ\sum_{j=1}^i\tau_1(W_j),
\end{equation}
for $\varphi\in\mc{K}_{g,1}={\mathcal M}_{g,1}[2]$.

\end{theorem}
\begin{proof}
First note that by definition we have
\begin{equation}\label{eq:tau1difference}
\tau_1^{GG'}(x)=\theta^G_2(\mb{x})-\theta^{G'}_2(\mb{x})
\end{equation}
and
\begin{equation}\label{eq:tau2difference}
\tau_2^{GG'}(x)=\theta^G_3(\mb{x})-\theta^{G'}_3(\mb{x})-( \tau_1^{GG'}\otimes 1+1\otimes\tau_1^{GG'}) \theta^{G'}_2(\mb{x}).
\end{equation}

It follows immediately that
\[
\\
\tau_1(\varphi)(x)=\sum_{i=1}^k \tau_1(W_i)(x)\!,
\]
and
\[
\begin{split}
\tau_2(\varphi)(x)&=\sum_{i=1}^k \ThG{3}{G_{i-1}}{x}  -\ThG{3}{G_i}{x}  \\
&=\sum_{i=1}^k \tau_2^{G_{i-1}G_i}(x)+( \tau^{G_{i-1}G_i}_1\otimes1+1\otimes\tau^{G_{i-1}G_i}_1)\ThG{2}{G_i}{x}\\
&=\sum_{i=1}^k \tau_2(W_i)(x)-\left( \tau_1(W_i)\otimes1+1\otimes\tau_1(W_i)\right)\circ\sum_{j=1}^i\tau_1(W_j)(x)
\end{split}
\]
since the sum
$
\sum_{i=1}^k (\tau_1(W_i)\otimes1+1\otimes\tau_1(W_i) )\theta^G_2(\mb{x})
$
telescopes to zero for $\varphi\in\mc{K}_{g,1}$.  
\end{proof}

\section{Invariant Tensors}
\label{sect:invarianttensors}

We now turn towards finding explicit representatives of the fatgraph Johnson maps. 
 First, we fix a Whitehead move $W:G\ra G'$ and denote   $\tau^{GG'}$ simply by $\tau$.  Write $\theta$ and $\theta'$ for $\theta^{G}$ and $\theta^{G'}$, respectively, and set $\ell=\log(\theta)$ and $\ell'=\log(\theta')$.   

Using this notation, we can rewrite  the definition of $\tau=\tau^{GG'}$ as $\tau \theta'=\theta$, which is equivalent to the relation $\tau \ell'=\ell$.  Since 
 $h(\cdot  ,\cdot  )_{(1)} = 0$, \eqref{eq:elln} gives 
\[
\ell'(\mb{x}) = x -(1/3)\sum^{k'}_{i=1} h(\ell'(\f'_{i-1}), -\ell'(\f'_i)).
\]

Combining these relations and using the fact that $\tau$ is an algebra automorphism, we have 
\[
\begin{split}
\ell (\mb{x}) = \tau\ell'(\mb{x})
&= \tau(x) -(1/3)\sum^{k'}_{i=1} h(\tau\ell'(\f'_{i-1}), -\tau\ell'(\f'_i))\\
&= \tau(x) -(1/3)\sum^{k'}_{i=1} h(\ell(\f'_{i-1}), -\ell(\f'_i)).
\end{split}
\]
The sum in the last line is taken over an edge-path $\{\f'_i\}$  in $G'$ while the summands involve the iterated integrals $\ell$ of $G$.  

It is important to note that $\ell(\f'_i)$ is well-defined for all $\f'_i$ in $G'$ because $\ell$ is defined on the whole $\pi_1$ and $\f'_i \in \pi_1$.  On the other hand, most oriented  edges $\f'_i$ in the edge-path of $G'$ can be naturally associated to oriented edges of $G$ and hence have natural combinatorial interpretations in terms of $G$.  

Define
\[
\tilde\ell(\mb{x}) := x -(1/3)\sum^{k'}_{i=1} h(\ell(\f'_{i-1}), -\ell(\f'_i)),
\]
for any oriented edge $\mb{x}$ of $G$, and extend as just discussed
for any element $\mb{x}\in \pi_1$.

We can then summarize the results of this section:
\begin{theorem}\label{thm:tauisdeltaell}
$\tau^{GG'}(x)=x+\ell(\mb{x})-\tilde\ell(\mb{x})$,
for any  $\mb{x}\in\pi_1$, and in particular, we have
\[
\tau^{GG'}_{m}(x)=\ell_{m+1}(\mb{x})-\tilde\ell_{m+1}(\mb{x})
\]
for all $m\geq 1$.
\end{theorem}

\subsection{Tensor representatives}\label{sec5}

We now wish to use Theorem \ref{thm:tauisdeltaell} to give explicit tensor formulae for the Johnson maps $\tau_k^{GG'}$ in terms of the combinatorics and  $H$-marking of $G$.  
For this, we consider the Whitehead move on an edge $\mb{e}$ as shown in Figure \ref{fig:caseI} and consider the effects of $\tau^{GG'}$ on some edge $\mb{x}\neq\mb{e}$ of $G$.  Since the $\pi$-markings on the collection of edges other than $\mb{e}$ generate $\pi_1$, this will suffice to determine $\tau^{GG'}$ completely.

If $\mb{x}\neq\mb{t}$, we shall without loss of generality assume that the edge-path connecting $\mb{x}$ to $\bar{\mb{x}}$ avoids the tail $\mb{t}$, while if $\mb{x}=\mb{t}$, then the edge-path is precisely the boundary cycle starting at $\mb{t}$.   This edge-path may pass through any of the four sectors surrounding the edge $\mb{e}$ in any order.  We shall label the sectors  $(\mb{d}, \mb{e}, \mb{\bar{a}})$, $(\mb{a},\bar{\mb{b}})$, $(\mb{b},\bar{\mb{e}},\bar{\mb{c}})$, and $(\mb{c},\bar{\mb{d}})$ by  I, II, III, and IV, respectively, as in Figure~\ref{fig:caseI}.

\begin{figure}[!h]
\begin{center}
\epsffile{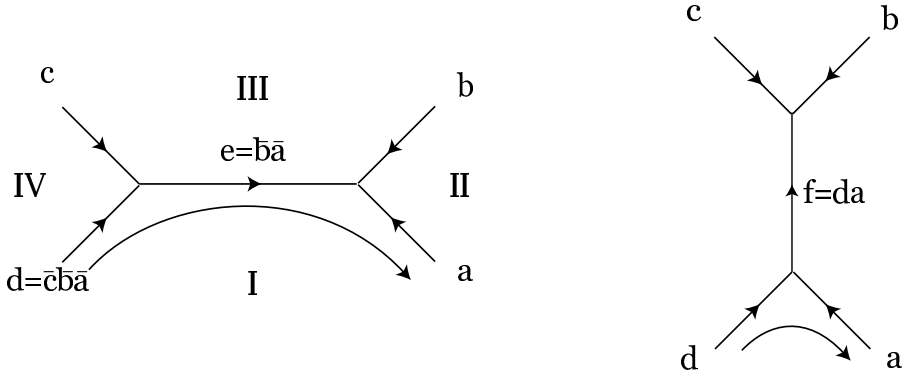} 
\caption{A Whitehead move.}
\label{fig:caseI}
\end{center}
\end{figure}

By Theorem \ref{thm:tauisdeltaell}, the value of $\tG_m(x)$  is given by  the  difference of two  sums $\ell_{m+1}(\mb{x})$ and $\tilde \ell_{m+1}(\mb{x})$ taken over edge paths of  $G$ and $G'$ respectively.
Since the summands in each case are the same away from the edge $\mb{e}$, these ``non-local'' contributions will cancel.  Thus, the difference will only pick up a contribution as the edge-path passes through one of the four sectors.  Let  us denote each of these contributions by $\tG_m(I)$, $\tG_m(II)$, $\tG_m(III)$, and $\tG_m(IV)$, so for example, if the path passes only through the sectors  $(\mb{a},\bar{\mb{b}})$ and $(\mb{b},\bar{\mb{e}},\bar{\mb{c}})$, then  $\tG_m(x)=\tG_m(II)+\tG_m(III)$.  
   
 Now consider the case $\mb{x}=\mb{t}$.  The path connecting $\mb{t}$ to $\bar{\mb{t}}$ must pass through all four sectors, and $\tG_m(t)=0$ since the tail is $H$-marked zero and  $\tG_m$ is by definition linear in $H$.  These two facts together immediately  imply 
    \begin{lemma}\label{lem:locality}
    \[
   0=\tG_m(I)+\tG_m(II)+\tG_m(III)+\tG_m(IV).
 \]
   \end{lemma}
   While seemingly simple, this lemma fundamentally depends upon the ``linearity with respect to sectors'' of our fatgraph Johnson maps which is provided by  
  Theorem \ref{thm:tauisdeltaell}, and this relation will allow us to find an explicit tensor representation for $\tG_m$.

\begin{lemma}  For a Whitehead move labeled as in Figure \ref{fig:caseI}, $\tG_m$ can be represented as  
    \begin{equation}\label{eq:sumofcases}
  \begin{split}
  \tG_m
  =-a\otimes \tG_m(I) - b\otimes (\tG_m(II)+\tG_m(I)) + c\otimes  \tG_m(IV) .
  \end{split}
  \end{equation}  
\end{lemma}
\begin{proof}
We exploit the fact that 
the intersection pairing of oriented edges is also linear with respect to sectors and restrict our attention to the four basic  cases where the edge path connecting $\mb{x}$ to its inverse traverses only one  sector.  
In each of the four cases, we can determine the homology intersection products $x\cdot a$, $x\cdot b$, and $x\cdot c$ directly using Proposition \ref{geomh} (and from these, one can also determine $x\cdot d$, $x\cdot e$, or $x\cdot f$ by linearity).  In particular for case I, these intersection numbers are $1$, $0$, and $0$ respectively.  Similarly, we can list all the intersection vectors as 
\begin{equation}\label{eq:intnums}
 \begin{pmatrix}
 x\cdot a \\ x\cdot b \\ x\cdot c 
 \end{pmatrix} =
   \begin{pmatrix}
~~1\\ ~~0\\ ~~0 
 \end{pmatrix},
   \begin{pmatrix}
-1\\ ~~1\\~~ 0
 \end{pmatrix},
   \begin{pmatrix}
 ~~0\\-1\\ ~~1
 \end{pmatrix}, \textrm{ and } 
    \begin{pmatrix}
~~0\\ ~~0 \\ -1
 \end{pmatrix}
\end{equation}
 for cases I-IV, respectively.  
  
  Combining this with Lemma \ref{lem:locality} we find that 
 \[
 \tG_m(x)
 =(x\cdot a) \tG_m(I) + (x\cdot b) (\tG_m(II)+\tG_m(I)) - (x\cdot c) \tG_m(IV) 
 \]
 gives a solution for all $\mb{x}$ regardless of the number or order of sectors traversed along the path from $\mb{x}$ to $\mb{\bar{x}}$.  Using the \Poin duality $H^*\cong H$  of $\S1g$ given by  $(a \cdot \;\;)\mapsto a$ (our sign convention matches that of \cite{Kawazumi05} but differs from that of Morita and Johnson), we obtain \eqref{eq:sumofcases}. 
\end{proof}
    
\begin{theorem}\label{thm:tautensor}
For a  Whitehead move $G\ra G'$  labeled as in Figure \ref{fig:caseI},  $\tau^{GG'}_m$ lies in  $H\otimes \El_{m+1}$ and is explicitly given by
\[
\begin{split}
3\cdot \tau^{GG'}_m&=a\otimes h(\ell(\mb{b}), \ell(\mb{c}))_{(m+1)} \\
&+b\otimes \left( h(\ell(\mb{b}), \ell(\mb{c}))-h(\ell(\mb{c}), \ell(\mb{d})) \right)_{(m+1)} \\
&+c\otimes h(\ell(\mb{a}), \ell(\mb{b}))_{(m+1)} .
\end{split}
\]
\end{theorem}
\begin{proof}
The theorem follows from \eqref{eq:sumofcases} by finding explicit  tensor descriptions for the individual cases $\tau_m^{GG'}(I-IV)$.  By symmetry, it suffices to consider only case I.  

As mentioned above, the combinatorial integrals defining $\ell_n(\mb{x})$ and $\tilde \ell_n(\mb{x})$ over $G$ and $G'$ will be mostly the same, and the discrepancy in case I occurs with  the partial path $(\mb{d},\mb{e},\mb{\bar{a}})$ in $G$ replaced by $(\mb{d},\mb{\bar{a}})$ in $G'$.  
We wish to reinterpret this difference in terms of the Hausdorff series and start by writing
\[
\tau^{GG'}(I)=x-\frac13\left( h(\ell(\mb{d}),-\ell(\mb{e}))+h(\ell(\mb{e}),\ell(\mb{a}))-h(\ell(\mb{d}),\ell(\mb{a}))  \right).
\]

Again, we simply write  $X*Y$ for $\log(\exp X \exp Y)$ if  $X,Y \in \mathcal{L}$. Associativity of our fatgraph Magnus expansion leads to 
\[
h(X, -Y) + h(Y, Z) + h(X*(-Y), Y*Z) = h(X, Z),
\]
for any $X, Y, Z \in \mathcal{L}$, since $X*(-Y)*Y*Z = X*Z$.  In the notation of Figure \ref{fig:caseI}, we take
\[
X = \ell(\mb{d}), \; Y = \ell(\mb{e}), \textrm{ and } Z = \ell(\mb{a})
\]
so that
 \[
X*(-Y) = -\ell(\mb{c}), \textrm{ and }  Y*Z = -\ell(\mb{b}),
\]
and hence
\[
\tau^{GG'}(I) = x+ \frac13 h(X*(-Y), Y*Z) = \frac13 h(-\ell(\mb{c}), -\ell(\mb{b})).
\]
By the (signed) symmetry of our equations under the permutation $\mb{a}\mapsto \mb{b}\mapsto \mb{c}\mapsto \mb{d}\mapsto \mb{a}$, we can solve for the other cases and find
\begin{equation}\label{eq:tauI}
\begin{split}
\tau^{GG'}_m(I) &=  -\frac13 h(\ell(\mb{b}), \ell(\mb{c}))_{(m+1)}, \\
\tau^{GG'}_m(II) &= \phantom{-}\frac13 h(\ell(\mb{c}), \ell(\mb{d}))_{(m+1)},\\
\tau^{GG'}_m(III) &= -  \frac13 h(\ell(\mb{d}), \ell(\mb{a}))_{(m+1)},\\
\tau^{GG'}_m(IV) &= \phantom{-}\frac13  h(\ell(\mb{a}), \ell(\mb{b}))_{(m+1)},
\end{split}
\end{equation}
where we have used the identity $h(-Y,-X)=-h(X,Y)$.  
\end{proof}

It is important to note that though this expression is entirely canonical, it is not unique
(cf. the next section).

\subsection{Explicit formulae}\label{sec51}
\label{sect:explicitformulae}

In order to simplify Theorem \ref{thm:tautensor}, we define the \emph{integral} iterated integrals  $P_x=6\cdot \ell_2(\mb{x})$, $Q_x=36\cdot  \ell_3(\mb{x})$, and $R_x=216\cdot \ell_4(\mb{x})$ for $\mb{x}\in \mathcal{E}_{or}(G)$.  More explicitly,
\begin{gather*}
P_x=\sum_{i=1}^k [f_{i-1},f_i], \\
Q_x=\sum_{i=1}^{k} [f_{i-1},[f_{i-1},f_i]] + [f_{i},[f_{i-1},f_i]] +[f_{i-1},P_{f_i}]+[P_{f_{i-1}},f_{i}], \\
\begin{split}
R_x=\sum_{i=1}^{k} &\big(  \, 3[f_i,[f_{i-1},[f_{i-1},f_i]]]\\
&+[f_{i-1},[f_{i-1},P_{f_i}]]+[f_{i-1},[P_{f_{i-1}},f_{i}]]+[P_{f_{i-1}},[f_{i-1},f_i]]\\
&+[f_{i},[f_{i-1},P_{f_i}]]+[f_{i},[P_{f_{i-1}},f_{i}]]+[P_{f_{i}},[f_{i-1},f_i]]\\
& +[P_{f_{i-1}}P_{f_i}]+[f_{i-1},Q_{f_i}]+[Q_{f_{i-1}},f_i] \big),\\
\end{split} 
\end{gather*}
where the sums are over the edge path $\{\f_i\}$ in $G$ connecting $\mb{x}$ to its reverse if such a path exists, and $P_x=-P_{\bar x}$ (resp.\ $Q_x=-Q_{\bar x}$, $R_x=-R_{\bar x}$) 
otherwise.

In terms of the notation of  Figures \ref{fig:vertex} and \ref{fig:caseI}, one can easily verify the following relations (which essentially hold by construction), 
\begin{subequations}\label{eq:relations}
\begin{align}\label{eq:Hrelations}
a+b+e&=0,\\
\label{eq:HWrelations}
a+b+c+d&=0,\\
\label{eq:Prelations}
P_a+P_b+P_e&=-3[a,b],\\
\label{eq:PWrelations}
\!\! P_a+P_b+P_c+P_d&=-3[a,b]-3[c,d],\\
\label{eq:Qrelations}
\!\! Q_a+Q_b+Q_e&=-3\left( [P_a,b]+[a,P_b]+[a,[a,b]]-[b,[a,b]]  \right).
\end{align}
\end{subequations}

As remarked following Theorem \ref{thm:tautensor}, the above expression is not the
unique Magnus expansion satisfying the conditions of Theorem \ref{thm:fatgraphmagnus}.  For example, 
\[
\hat Q_x=\sum_{i=1}^{k} [f_{i-1},P_{f_i}]+[P_{f_{i-1}},f_{i}] 
\]
also satisfies \eqref{eq:Qrelations}, and the assignment $\ell_3(\mb{x})=\frac{1}{36} \hat Q_x$ defines another such Magnus expansion.

  Now returning to Theorem \ref{thm:tautensor} for $m=1$,  we  obtain
\begin{equation}\label{eq:t1rough}
\begin{split}
\tG_1&=\sixth \Big( a\otimes[b,c] + b\otimes  ([d,c]+[b,c] )+ c\otimes  [a,b] \Big)\\
&=\sixth (a\otimes [b,c] +b\otimes [c,a] + c\otimes [a,b] ),
\end{split}
\end{equation}
where we have simplified using \eqref{eq:HWrelations}.
Under the bracket map $H\otimes\mathcal{L}_2\ra \mathcal{L}_3$, the above expression vanishes by the Jacobi identity, and we conclude that $\tG_1\in \mathcal{H}_1$.  

For $\tau_2$, 
 \eqref{eq:tauI} and \eqref{eq:haus} allow us to determine 
\[
\begin{split}
36\, \tG_2(I)=&-([b,P_c]+[P_b,c]+[b,[b,c]]-[c,[b,c]]),\\
36\, \tG_2(II)=& \;\;\quad [c,P_d]+[P_c,d]+[c,[c,d]]-[d,[c,d]],\\
36\, \tG_2(III)=&-([d, P_a]+[P_d,a]+[d,[d,a]]-[a,[d,a]]), \textrm{ and }\\
36\, \tG_2(IV)=& \;\;\quad [a,P_b]+[P_a,b]+[a,[a,b]]-[b,[a,b]],
\end{split}
\]

\noindent and using \eqref{eq:Hrelations} through \eqref{eq:PWrelations}, one  can directly check that 
\[ 
\begin{split}
36 \tG_2(II)+36\tG_2(I)= [a,P_c]-[c,P_a]-4[a,[b,c]]-[a-2b-c,[a,c]].
\end{split}
\]
Thus by \eqref{eq:sumofcases}, we have 
\[
\begin{split}
36\, \tG_2&=a\otimes ([b,P_c]-[c,P_b]+[b-c,[b,c]])\\
&+b\otimes( [c,P_a]-[a,P_c]-4[a,[b,c]]-[a-2b-c,[a,c]])\\
&+c\otimes ([a,P_b]-[b,P_a]+[a-b,[a,b]]),\\
\end{split}
\]
which does not vanish under the bracket map in general.  Thus, the image of
$\tG_2$ does not lie in $\mathcal{H}_2$ as one might have hoped, cf. the next section.

For $\tG_3$, an implementation of \eqref{eq:haus} and \eqref{eq:relations} using Mathematica \cite{Wolfram} on an Apple Powerbook  provides the expression
\[
\begin{split}
216\tau_3(x)&=a\otimes([b, Q_c] + [b, [b, P_c]] - 2 [b, [c, P_b]] + [b, [c, P_c]] - 3 [b, [c, [b, c]]] \\
&\quad\quad- [c, Q_b] + [c, [b, P_b]] - 2 [c, [b, P_c]] + [c, [c, P_b]] + [P_b, P_c])\\
&+b\otimes (-[a, Q_c] - [a, [a, P_c]] - 6 [a, [a, [b, c]]] - 4 [a, [b, P_c]] \\
&\quad\quad + 6 [a, [b, [a, c]]] - 6 [a, [b, [b, c]]] + 2 [a, [c, P_a]] + 2 [ a, [c, P_b]] \\
&\quad\quad- [a, [c, P_c]] + 3 [a, [c, [a, c]]] + 2 [b, [a, P_c]] + 6 [b, [a, [b, c]]] + 2 [b, [c,  P_a]] \\
&\quad\quad+ [c, Q_a] - [c, [a, P_a]] + 2 [c, [a, P_b]] + 2 [c,  [a, P_c]] + 6 [c, [a, [b, c]]] \\
&\quad\quad-  4 [c, [b, P_a]] - [c, [c, P_a]] - [P_a, P_c])\\
&+c\otimes ([a, Q_b] + [a, [a, P_b]] - 2 [a, [b, P_a]] + [a, [b, P_b]] - 3 [a, [b, [a, b]]] \\
&\quad\quad- [b, Q_a] + [b, [a, P_a]] - 2 [b, [a, P_b]] + [b, [b, P_a]] + [P_a, P_b]),\\
\end{split}
\]
which again in general  does not lie in $\mathcal{H}_3$.  (This is the unique calculation in this Section 5
that is computer assisted.)

We finally note that these expressions simplify considerably if any of the four edges $a,b,c,$ or $d$ are $H$-marked zero, and furthermore, every generator of ${\mathcal K}_{g,1}={\mathcal M}_{g,1}[2]$ can be represented by a composition of such Whitehead moves, cf. Section 6.3.

For example, if $b=0$, one can check that we have 
\begin{subequations}\label{eq:tb0}
\begin{align} \label{eq:t1b0}
\tG_1&=0,\\
 \label{eq:t2b0}
  \tG_2&=\frac{1}{6^2} \left(c\otimes [a,P_b]  -a\otimes [c,P_b]  \right), \\
 \label{eq:t3b0} 
 \tG_3 &= 
\frac{1}{6^3}\big(c\otimes ([a, Q_b] +[P_a, P_b]+[a, [a, P_b]] ) \\  \notag
&\quad\quad-a\otimes([c, Q_b] + [P_c, P_b]- [c, [c, P_b]] )\big), \textrm{ and }\\
 \label{eq:t4b0} 
 \tG_4 &= 
\frac{1}{6^4}\big(c\otimes ([a,R_b] + [P_a,Q_b]+ [ Q_a,P_b]+ [a, [a, Q_b]] )\\  \notag
&\quad\quad-a\otimes([c,R_b] + [P_c,Q_b]+ [ Q_c,P_b]- [c, [c, Q_b]] ) \big).
\end{align}
\end{subequations}

\subsection{Cochain representatives of $\tau_1$ and $\tau_2$}
\label{sect:cochainreps}

Using the natural isomorphism $\mathcal{H}_2\cong\Lambda^3 H$ (see \cite{Morita89}),  \eqref{eq:t1rough} can be more elegantly written
\[
\tG_1=\frac16 a\wedge b\wedge c.
\]
Thus, $\tG_1$ defines a 1-cochain in $C^1(\mathcal{\hat G_T}, \Lambda^3 H)$ which matches the cochain ${1\over 6}j_1$ of Morita-Penner \cite{moritapenner}.  By our construction, $\tG_1=\frac16 j_1$ is  an $\M1g$-equivariant cocycle which represents  $\tau_1$ when restricted to the Torelli group.
Morita-Penner proved these properties of $j_1$ by quite different means (as was discussed in the Introduction).

We now turn towards defining a similar cochain $j_2$ representing $\tau_2$ and begin by recalling some aspects of Morita's \cite{Morita96} group cohomological  lift  $\tilde k_2$. Recall (see \cite{Morita89})  the mapping $\varpi:  \Lambda^2 H \otimes \Lambda^2 H \ra  H\otimes \mathcal{L}_3(H)$ given by the composition of the inclusion $ \Lambda^2 H \otimes \Lambda^2 H \subset H\otimes H \otimes \Lambda^2 H$ and the projection $H\otimes H \otimes \Lambda^2 H \rightarrow H\otimes (H \otimes \Lambda^2 H)/\Lambda^3 H \cong H\otimes \mathcal{L}_3(H)$, which is explicitly given by
\[
\varpi(a \wedge b \otimes c \wedge d ) = a \otimes [b,[c,d]] - b \otimes [a,[c,d]].
\]
One can check  that the  image of the symmetric square $S^2(\Lambda ^2 H)$ of $\Lambda^2 H$ under $\varpi$ lies in $\mathcal{H}_2$, and in fact,  the image of Johnson's second homomorphism $\tau_2$ can be considered (modulo torsion) as a quotient of $S^2(\Lambda^2 H)$  by \cite{Morita89}. 
  Following  \cite{Morita89}, we shall write $t_1^{\otimes 2}$ and $t_1\leftrightarrow t_2\in S^2(\Lambda^2 H)$ for $t_1\otimes t_1$ and $t_1\otimes t_2 + t_2\otimes t_1$, respectively, where $t_i\in \Lambda^2 H\cong \mathcal{L}_2$.  By abuse of notation, we shall use the same notation to denote the corresponding  images of $t_1\leftrightarrow t_2$ and $t_1^{\otimes 2}$  in $\mathcal{H}_2$ under $\varpi$.

In \cite{Morita91}, Morita (using a different sign convention for \Poin duality $H\cong H^*$) defined a 
skew-symmetric pairing  $\cdot :\Lambda^3 H \times \Lambda^3H\ra S^2(\Lambda^2H)$ by (the negative of)
\begin{equation}\label{eq:moritaproduct}
\begin{split}
-2~(a\wedge  b\wedge c)  &\cdot (d\wedge e\wedge f)  =  \\
 & (a\cdot d)  [b,c] \leftrightarrow [e,f] + (a \cdot e) [b,c] \leftrightarrow [f,d]  + (a\cdot f) [b,c] \leftrightarrow [d,e] \\
& (b\cdot d)  [c,a] \leftrightarrow [e,f] + (b \cdot e)  [c,a]  \leftrightarrow [f,d]  + (b\cdot f)  [c,a]  \leftrightarrow [d,e] \\
& (c\cdot d)  [a,b] \leftrightarrow [e,f] + (c \cdot e) [a,b] \leftrightarrow [f,d]  + (c\cdot f) [a,b] \leftrightarrow [d,e] . \\
\end{split}
\end{equation}
Writing $\chi(\xi,\eta)=-2( \xi\cdot \eta)$, for $\xi,\eta\in\Lambda^3 H$, this product defines an Euler class $\chi\in H^2(\Lambda^3 H, \mathcal{H}_2)$  of a corresponding central extension $$1\ra\mathcal{H}_2 \ra\mathcal{H}_2 \tilde\times \Lambda^3H \ra  \Lambda^3 H  \ra 1$$ of $\Lambda^3 H$ which encodes the action (modulo torsion) of $\M1g[1]$ on the third nilpotent quotient $N_3$.  In \cite{Morita96} Morita showed 
  how this leads to  the existence of a crossed homomorphism $[\tilde k_2]\in H^1(\M1g,\mathcal{H}_2 \tilde\times \Lambda^3H)$ which lifts  $\tau_2$ to all of $\M1g$.  This completes the synopsis of Morita's
results which we shall require here.

The canonical cocycle $j_1$ representing the first Johnson homomorphism was first computed
in \cite{moritapenner}, and
 the question was posed there as to whether there exists a canonical cochain $j_2\in Z^1_{\M1g}(\mathcal{\hat G},\mathcal{H}_2\tilde \times \Lambda^3H)$ in the dual  fatgraph complex which similarly represents $\tau_2$.  We shall next give an affirmative answer to this question.  

The main idea is to ``symmetrize'' our expression for $\tG_2$ so that its image lies in $\mathcal{H}_2$.  To this end, consider the map
$\,{}\bar{} :H\otimes \El_3\ra\mathcal{H}_2$ defined by 
\begin{multline*}
x\otimes [y,[z,w]] \mapsto  \frac14 [x,y]\lra [z,w] \\
=\frac14 \left(x\otimes[y,[z,w]]-y\otimes[x,[z,w]]+z\otimes [w,[x,y]]-w\otimes [z,[x,y]]\right).
\end{multline*}
  One can check that this map is well defined and is the identity on $\varpi(S^2(\Lambda^2 H))$.  Composing with this projection, we obtain $\bar \tau_2^{GG'}\in\mathcal{H}_2$ with 
\[
\bar \tau^{GG'}_2=\frac{1}{2\cdot 36}\left([a,b]\lra P_c +[b,c]\lra P_a+[c,a]\lra P_b+3[a,b]\lra[b,c]\right),
\]
which simplifies to 
\begin{equation}\label{eq:tau2bzero}
 \btG_2= \frac{1}{2\cdot 36}[c,a]\lra P_b
\end{equation}
 when $b=0$.

We now reformulate \eqref{eq:localtau2sum}.  
If $\tau_1=a\otimes [b, c] + b\otimes  [c,  a]+ c\otimes [a,b]$ and $\tau'_1=d\otimes [e,f]+ e\otimes [f,d]+ f\otimes [d,e]$, then we have 
\[
\begin{split}
&(1\otimes \tau_1 +\tau_1\otimes 1)\circ \tau'_1 =
 a\cdot d \left(  f\otimes [[b,c],e] + e\otimes [f,[b,c]] \right) \\
&+ a\cdot e \left(  d\otimes [[b,c],f] + f\otimes [d,[b,c]] \right) 
+ a\cdot f \left(  e\otimes [[b,c],d] + d\otimes [e,[b,c]] \right) \\
&+ b\cdot d \left(  f\otimes [[c,a],e] + e\otimes [f,[c,a]] \right) 
+ b\cdot e \left(  d\otimes [[c,a],f] + f\otimes [d,[c,a]] \right) \\
&+ b\cdot f \left(  e\otimes [[c,a],d] + d\otimes [e,[c,a]] \right) 
+ c\cdot d \left(  f\otimes [[a,b],e] + e\otimes [f,[a,b]] \right) \\
&+ c\cdot e \left(  d\otimes [[a,b],f] + f\otimes [d,[a,b]] \right) 
+ c\cdot f \left(  e\otimes [[a,b],d] + d\otimes [e,[a,b]] \right). 
\end{split}
\]
Under the projection $H\otimes\mathcal{L}_3\ra \mathcal{H}_2$, this composition  matches the negative of  the pairing $\Lambda ^3H \times \Lambda^3 H\ra \mathcal{H}_2$ of Morita given in \eqref{eq:moritaproduct}.  In this formulation, \eqref{eq:localtau2sum} becomes
\[
\bar \tau_2(\varphi)=\sum_{i=1}^k \bar \tau_2(W_i)+\tau_1(W_i)\cdot\sum_{j=1}^i\tau_1(W_j),
\]
which in more suggestive iterated integral notation can be written 
\[
\bar \tau_2(\varphi)=\int_\varphi \bar \tau_2+\int_\varphi \tau_1\cdot\tau_1.
\]
As a consequence, the skew property of the pairing in the above formulae dictates that   $\btG_2=-\bar \tau_2^{G'G}$ for every Whitehead move $G\ra G'$ (while this is not true for $\tG_2$).  

Defining the integral cochain $j_2\in C^1(\hat{ \mathcal{G}}_T,\mathcal{H}_2\tilde\times \Lambda^3H)$ by 
\[
j_2(W)=72~ \, \bar \tau_2(W)\tilde{\times}\tau_1(W),
\]
we may summarize this section with the following result.

\begin{theorem}
The cochain $j_2$ is a canonically defined integral cocycle $j_2\in Z^1_{\M1g}(\mathcal{\hat G}_I,\mathcal{H}_2\tilde\times \Lambda^3H)$ which represents the second Johnson homomorphism on $\M1g[2]$ and  projects onto a multiple of $j_1$ under the natural projection map.  
\end{theorem}
\begin{proof}
Since the image of Johnson's second homomorphism $\tau_2$ lies in $\mathcal{H}_2$, we conclude
$\tau_2(\varphi)=\bar\tau_2(\varphi)$, for $\varphi\in \M1g [2]$.  The theorem follows from this observation and the definition of $j_2$.
\end{proof}

\subsection{Cochain representatives of  $\tau_m$}
\label{sect:morecochainreps}

The composition  of elements of $IA(\widehat T)$ leads to a natural product structure on $ Hom(H,\widehat T_2)$ (thus also on $H\otimes \mathcal{L}$)  such that  for  Whitehead moves $G\ra G'$ and $G'\ra G''$,   we have $ \tau^{GG''}=\tau^{GG'}+ \tau^{G'G''}+\tau^{GG'}\circ \tau^{G'G''}$ with, for example,
\begin{equation}\label{eq:compositions}
\begin{split}
(\tau^{GG'}\circ \tau^{G'G''})_{(4)}&=  (\tau_1^{GG'}\otimes 1+1\otimes \tau_1^{GG'}  )\circ  \tau_1^{G'G''},\\
(\tau^{GG'}\circ \tau^{G'G''})_{(5)}&=(\tG_2\otimes 1+ 1\otimes \tG_2 + \tau_1^{GG'} \otimes \tau_1^{GG'} )\circ  \tau_1^{G'G''}\\ 
&+(\tau_1^{GG'}\otimes1 \otimes 1+1\otimes \tau_1^{GG'}\otimes 1 + 1\otimes 1\otimes \tau_1^{GG'}  )\circ  \tau_2^{G'G''},
\end{split}
\end{equation}
where we consider $H\otimes\mathcal{L}_k$ to be of degree $k+1$.
Each summand in this expression in degree four arises by ``mixing'' $\tau_1$ with itself, and terms in this degree-five expression arise by likewise mixing $\tau _1$ with itself and by mixing $\tau _1$ with $\tau _2$.  The expression for $(\tau^{GG'}\circ \tau^{G'G''})_{(m)}$ in general degree $m$ arises by similarly mixing lower degree groupoid Johnson homomorphisms.

In this way, our fatgraph Johnson maps $\tau^{GG'}$ can be considered (in a loose sense) as 1-cocycles on $\mathcal{\hat G}_I$ with non-abelian coefficients $H\otimes \mathcal{L}$.  This terminology is warranted  in that their values are invariant under (based) homotopy of paths (or loops) in $\mathcal{\hat G}_I$.  However, they are restricted by the feature that their values do depend on the initial basepoints of such paths (or loops).  

Refining the above,  let  $\mathcal{L}_{\leq k}=\coprod_{i=1}^k \mathcal{L}_i$ denote the truncated free Lie algebra on $H$ and give $H\otimes \mathcal{L}_{\leq k}$ the natural product descending   from that of $H\otimes \mathcal{L}$ by discarding terms of degree higher that $k$.   Similarly, we define $\tG_{\leq m}=\sum_{i=1}^m \tG_i$.  In this way, the $\tG_{\leq m}$ are (basepoint dependent) 
 1-cocycles with coefficients in  $H\otimes \mathcal{L}_{\leq m+1}$ that represent $\tau_m$ (independently  of basepoint) when restricted to $\M1g[m]$. 

 Thus in some sense, every $\tau_m$ can be lifted to a 1-cochain in $\mathcal{\hat G}_I$.  
   However, these lifts $\tau _m$, for $m>2$, are  less satisfactory than those of $\tau_1,\tau_2$ in the previous section since the general $\tau_m$ has large cokernel (with our current limited understanding in terms of symplectic representation theory), and there is no known
interpretation as an iterated integral.  
On the other hand, it is always possible to ``symmetrize'' the coefficients of $\tG_m$   as for $m=1,2$
to have image in $\mathcal{H}_m$, but the resulting expressions appear to be  quite complicated for $m>2$.  
Furthermore, the calculation of the general Johnson homomorphism is computationally intensive (much as are the crossed homomorphisms of the classical Magnus representations, cf. \cite{Suzuki02}).  It is worth emphasizing, however,  that with sufficient computer power, our results give a practical method for calculating $\tau_m(\varphi)$ for any reasonably small $m$ and any $\varphi$ which is conveniently represented as a sequence of Whitehead moves.

\section{Examples for twists on separating curves }
\label{sect:examples}

In this final section, we illustrate our techniques by calculating the 1-cochain Johnson homomorphisms $\tau_1$, $\tau_2$, and the fatgraph Magnus lift of $\tau_3$ for certain elements of the subgroup $\mathcal{K}_{g,1}\cong\M1g[2]$ generated by (necessarily infinitely many, see \cite{BissFarb}) Dehn twists on separating curves.  In fact, we shall evaluate these Johnson maps on certain sequences of Whitehead moves which serve as a natural extension
(in the sense of groupoids)  of the set of separating twists.   

We begin by considering  the right-handed Dehn twist $T_\partial$ on the boundary $\partial\S1g$ of $\S1g$, which  has the nice feature  that it is easily described in terms of Whitehead moves: it is represented by a sequence of repeated Whitehead moves on the edge ``to the right'' of the tail $\mb{t}$.  
This boundary Dehn twist is furthermore the basic case of Dehn twists on other separating curves in higher genus surfaces by the ``stability property'' of bordered mapping class groupoids, which we next formulate:

Suppose that $e$ is a separating edge of the $\pi$-marked trivalent fatgraph $G$, i.e., removing $e$ from $G$ produces a disconnected graph.
Consider the component $G'$ of $G-\{ e\}$ not containing the tail, so $G'$ inherits a $\pi$-marking by restriction, and hence a geometric $H$-marking, from that on $G$.  We may take $e$ as the tail of $G'$ and find that
by construction for any edge $x$ of $G'$, we have $\theta^ 
{G'}(x)=\theta^G(x)$.

For any sequence of Whitehead moves on edges  
of $G'$, we have an algorithm for computing $\tau_m$, for any $m\geq 1$, and
this clearly takes the same value as if we regarded the Whitehead moves
as occurring on $G$ itself.   Let $S(G)$ denote the once-bordered surface corresponding
to $G$ and likewise $S(G')$ for $G'$, so there is a natural inclusion $i:S(G)\to S(G')$
(in the usual sense of stability \cite{Harer85} of mapping class groups).  

The {\it groupoid stability
property} of bordered surfaces is that if a sequence of Whitehead moves on edges of $G'$ represents an element
$\varphi$ in the mapping class group of $S(G')$, then this same sequence of Whitehead
moves along edges of $G'$ in the fatgraph $G$ represents $i_*(\varphi )$ in the mapping class group of $S(G)$.

Although our fatgraph Magnus expansion is canonical, for calculational purposes, it will be convenient to consider one particular ``symplectic'' fatgraph $G_0$ as the basepoint in $\mathcal{\hat G}_I$.  We have depicted our chosen basepoint fatgraph with symplectic $H$-basis $\{u_i,v_i\}$ corresponding to a standard generating set $\{\mb{u}_i,\mb{v}_i\}$ of $\pi_1$ in Figure \ref{fig:sympg}.  For this fatgraph, the twist $T_\partial$ is represented by moving the tail $\mb{t}$ around the fatgraph  to the right until until it returns to its original position.  Similarly, we could for instance perform a genus two separating twist by using the groupoid stability property and moving the edge labeled $\mb{t'}$ around the genus two subgraph of which it is the tail.

\begin{figure}[!h]
\begin{center}
\epsffile{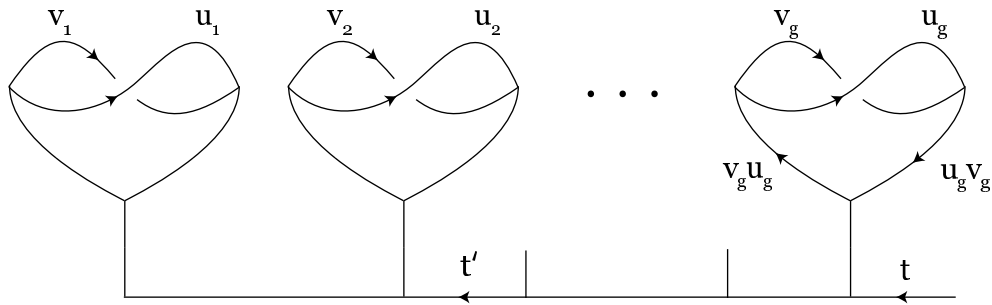}
\caption{The symplectic genus $g$  fatgraph $G_0$.}
\label{fig:sympg}
\end{center}
\end{figure}

Using Mathematica \cite{Wolfram} on an Apple Powerbook, we have calculated the fatgraph Magnus expansion  of $G_0$ and have the following modulo $\widehat T_5$:
\[
\begin{split}
\ell^{G_0}(\mb{u}_i)&= u_i + \frac12 [u_i, v_i ] - \frac19[u_i, [u_i, v_i ]] -\frac{1}{18} [v_i , [u_i, v_i ]] \\
& + \frac{1}{72} ([u_i, [u_i, [u_i, v_i ]]] + [u_i, [v_i , [u_i, v_i ]]]  - [v_i , [v_i , [u_i, v_i ]]]),\\
\ell^{G_0}(\mb{v}_i)&= v_i - \frac12 [u_i, v_i ] + \frac19[v_i , [u_i, v_i ]] +\frac{1}{18} [u_i, [u_i, v_i ]] \\
&+  \frac{1}{72} ([u_i, [u_i, [u_i, v_i ]]] - [u_i, [v_i , [u_i, v_i ]]] - [v_i , [v_i , [u_i, v_i ]]]), \textrm{ and } \\
\ell^{G_0}(\mb{t}) &= -\omega+ \frac{1}{36} \sum_{i=1}^g \left( [u_i, [u_i, [u_i, v_i ]]] + [u_i, [v_i ,[u_i, v_i ]]] + [v_i , [v_i , [u_i, v_i ]]]\right),\\
\end{split}
\]
 where $\omega= \sum_{i=1}^g [u_i,v_i]$ is the ``symplectic class''.

As a consequence, we have the following
\begin{proposition}\label{prop:tail}
For any trivalent marked bordered fatgraph $G$, 
\[
\theta^G(\mb{t}) = 1 + 0- \omega + 0 + \theta^G_4(\mb{t}) + \dotsm.
\]
\end{proposition}
\begin{proof}
The result holds for the symplectic fatgraph $G_0$ of Figure \ref{fig:sympg} by our calculation above.  Since $\tG_1(t)=0$ for any Whitehead move $G\ra G'$, \eqref{eq:tau1difference} implies that $\theta^G_2(\mb{t})=\theta^{G'}_2(\mb{t})$, thus, $\theta^{G}_2(\mb{t})=\omega$ for all fatgraphs $G$.  Similarly, $\tG_2(t)=0$,  and  $(\tG_1 \otimes 1 + 1 \otimes \tG_1) (\omega )=0$ since $\tG_1\in \mathcal{H}_1$ (see Lemma 4.5 of \cite{Morita93}), so that \eqref{eq:tau2difference} implies $\theta^G_3(\mb{t})$ is also unchanged by any Whitehead move and is thus always zero.
\end{proof}
Note that for $m>3$,  $\theta^G_m(\mb{t})$ does depend sensitively on the combinatorics of $G$.

\subsection{$\tau_1$ and $\tau_2$ on $\mathcal{K}_{g,1}$}\label{sec61}

Since the tail $\mb{t}$ has zero $H$-marking and every Whitehead move in the composition representing the twist $T_\partial$ on $\partial \S1g$ is along an edge adjacent to $\mb{t}$, we can compute the contributions to $\tau_1$, $\tau_2$ and $\tau_3$ of each move using \eqref{eq:tb0} and \eqref{eq:tau2bzero}.  For instance, \eqref{eq:t1b0} immediately implies that the contributions to $\tau_1$ all vanish, so $\tau_1(T_\partial)=0$ as expected.  

To compute $\tau_2$ for the twist $T_\partial$, we first note that since $\tau_1(W)=0$ for each Whitehead move $W$ in the composition representing $T_\partial$, \eqref{eq:localtau2sum} simplifies to give $\tau_2(T_\partial)=\sum_W\tau_2(W)$, where the sum here and below is over all the Whitehead moves in the composition representing $T_\partial$.  Moreover, since $P_t=-6\,\omega$, \eqref{eq:tau2bzero} shows that each Whitehead move will give a contribution of the form
\[
 \bar \tau_2(W) = \frac{1}{2\cdot 6}\left([a,c]\lra \omega\right).
\]
One can check that the sequence of (nontrivially contributing) values of $a$ and $c$ in the above formula produced by the Whitehead moves representing $T_\partial$ are given by $(a,c)=$
\begin{equation}\label{eq:accases}
\begin{array} {c c c}
(-u_i-v_i, u_i), & (-u_i,-v_i), & (v_i,-u_i-v_i), \\
(u_i+v_i, -u_i), & (u_i,v_i), & (-v_i,u_i+v_i), \\
\end{array}
\end{equation}
for $i=1,\ldots, g$.  Note that these values can be read off from the sectors at each vertex of $G$, and we have not included the pair $(u_i+v_i,-u_i-v_i)$ since its contribution is trivial.   In each of these cases, we have $[a,c]=[u_i,v_i]$ for some $i$.  Since there are six  cases for each $i$, the final sum gives
\begin{equation}\label{eq:tau2formula}
\begin{split}
 \tau_2(T_\partial)= \bar\tau_2(T_\partial)&=\sum_{i=1}^g  \frac 12 [u_i,v_i]\lra \omega= \omega^{\otimes 2}\\
&= \sum_{i=1}^g  u_i\otimes[v_i,\omega]-v_i\otimes[u_i,\omega]
 \end{split}
\end{equation}
which differs (because of our differing sign conventions) only in sign from  the result in \cite{Morita89}.

It is important to note that the above expression is independent of the choice of symplectic basis for the surface $\S1g$.  Indeed, this reflects the fact that the Johnson homomorphisms are $\M1g$-equivariant in the sense that $\tau_m(\varphi\psi\varphi^{-1})=|\varphi|(\tau_m(\psi))$, for $\varphi\in\M1g$ and $\psi\in\M1g[m]$.  In fact, the above expression is sufficient to calculate $\tau_2$ for any element of $\mathcal{K}_{g,1}$ which can explicitly be written as a product of separating twists.

\subsection{$\tau_3$ on $\mathcal{K}_{g,1}$}\label{sec62}
Again relying on  the fact that every Whitehead move $W$ in the composition representing the twist $T_\partial$ on $\partial \S1g$ has vanishing $\tau_1(W)$, \eqref{eq:compositions} shows that  $\tau_3(T_\partial)$ is simply equal to the sum  of $\sum _W \tau_3(W)$ as before.  
By Proposition \ref{prop:tail} and \eqref{eq:t3b0}, each such contribution will be of the form 
\begin{equation}\label{eq:t3septwist}
 \tau_3(W) = 
\frac{\!-1}{36}\big(c\otimes ([P_a, \omega]+[a, [a, \omega]] )
-a\otimes( [P_c,\omega]-[c, [c, \omega]] )\big).
\end{equation}
One can check that the contributions of the $c\otimes [a, [a, \omega]] $ terms given by the sequence \eqref{eq:accases} cancel, as do  the  $a\otimes [c, [c, \omega]]$ terms. Also, the two contributions of the pair $(a,c)=(u_i+v_i,-u_i-v_i)$ not listed in \eqref{eq:accases} cancel.  

To compute the contributions of the remaining terms, it is necessary to know the values of $P_x$ for certain oriented edges $\mb{x}$ of each fatgraph arising in the sequence of  Whitehead moves in the composition representing $T_\partial$.  
Fortunately, these values are easy to derive.  

Consider first the fatgraph $G_0$.  We have already observed  above that the bracket of edges at every sector listed in \eqref{eq:accases} takes the value   $-[u_i,v_i]$, for some $i$ (with the correct orientation according to our combinatorial integrals).  Thus, the value of $P_x$ for an oriented edge $\mb{x}$ is determined solely by the number of sectors (for each $i$) which are traversed in the path from $\mb{x}$ to $\mb{\bar x}$ (or its reverse).  It follows that every edge $\mb{x}$ with nonzero $H$-marking has $P_x=\pm 3[u_i,v_i]$, for some $i$, while those edges with $x=0$ have  either $P_x=\pm 6[u_i,v_i]$, for some $i$, or   $P_x=\pm \sum_{i=1}^h 6[u_i,v_i]$, for some $h$.  Moreover, the signs are all negative if the path from $\mb{x}$ to $\mb{\bar x}$ avoids the tail.  

Consider an arbitrary fatgraph arising in the sequence of Whitehead moves representing $T_\partial$.  Such a fatgraph resembles $G_0$ but with the tail attached at a different location.  In this case, it is possible that the path connecting $\mb{x}$ to $\mb{\bar x}$ is forced to go ``the long way'' around the boundary cycle.  The net result is that we get $P_t=-6\omega$ minus the ``expected'' value of $P_x$.  Since the value of $P_x$ only enters \eqref{eq:t3septwist} through the terms $[P_a, \omega]$ and $[P_c, \omega]$, we only require the value of $P_x$ modulo $\omega$, and the consequence of going ``the long way'' reduces to a change of sign.   

With this recipe, one can check that the during the sequence of Whitehead moves in the composition representing the twist $T_\partial$, the possible values  for the $a$ and $P_c$  in \eqref{eq:t3septwist} are given (modulo $\omega$) by $(a,P_c)=$
\[
\begin{array} {c c c}
(-u_i-v_i, 3[u_i,v_i]), & (-u_i,3[u_i,v_i]), & (v_i,-3[u_i,v_i]), \\
(u_i+v_i,3[u_i,v_i]), & (u_i+v_i, -3[u_i,v_i]), & (u_i,-3[u_i,v_i]), \\
 (-v_i,-3[u_i,v_i]), & \text{ and }& (-u_i-v_i,-6[u_i,v_i]),
\end{array}
\]
for $i=1,\ldots, g$.  Summing the resulting  values of $a\otimes [P_c,\omega]$ for these pairs gives $(v_i-u_i)\otimes [3[u_i,v_i],\omega]$.  Similarly, the sum of values of $c\otimes [P_a,\omega]$  gives the same  $(v_i-u_i)\otimes [3[u_i,v_i],\omega]$, and the total contribution of  \eqref{eq:t3septwist} over the twist $T_\partial$ is thus zero:
\begin{equation}\label{eq:tau3formula}
\tau_3(T_\partial)=0.
\end{equation}

We now wish to derive a method for calculating $\tau_3$ on twists on separating curves other than $\partial\S1g$.   Let $\eta$ be such a curve which separates a subsurface of genus $h$, and let $\varphi_\eta$ be an element of $\M1g$ which takes  $\eta$ to the standard  genus $h$ separating curve $\partial_h$ on our surface $\S1g$  (with respect to its symplectic generators  $\{\mb{u}_i,\mb{v}_i\}$).   Next, let $\gamma$ be a sequence of Whitehead moves which represents $\varphi_\eta$ and let  $\tau_{\leq 3}(\gamma)$ denote the value of $\tau_{\leq 3}$ on this sequence.   We can then use \eqref{eq:compositions} to compute the value of $\tau_3$ on the twist $T_\eta$ by composing $\tau_{\leq 3}(\gamma)$, $ \tau_{\leq 3}(T_\eta)$, and $ \tau_{\leq 3}(\gamma^{-1})$.

Before we discuss how to perform this calculation, we mention that it is not necessarily clear if there is a concise algorithm for translating a given element of $\M1g$ into a sequence of Whitehead moves (however, the converse is straightforward by the evolution of $\pi$-markings).  Because of this, we then consider the following generating set of $\mathcal{K}_{g,1}$.  

Let $\Gamma$ denote the set of paths $\gamma$ in $\mathcal {\hat G}_I$ beginning at $G_0$ and ending at a combinatorially isomorphic  fatgraph (with a possibly different $H$-marking).  Consider the set  $\{(\gamma, h) \; |\;  \gamma\in \Gamma,  h\in \{1, \ldots , g\}\}$ of loops in $\mathcal{\hat G}_I$ which are defined by traversing $\gamma$, performing a (standard) genus $h$ twist, then traversing $\gamma^{-1}$.  Since this set  obviously contains the set  of twists on separating curves of $\S1g$ (in fact, it is essentially the set of such twists with  redundancies), it serves as a generating set for $\mathcal{K}_{g,1}$ which is more natural in our groupoid setting.  

Now we return to compute the value of $\tau_3$ on an element $(\gamma, h)$.  First, note that (while in general $\tau_2(\gamma)\neq \tau_2(\gamma^{-1})$), if we simply traversed $\gamma$ followed by  $\gamma^{-1}$, we would obviously have $\tau_1=\tau_2=\tau_3=0$.  Moreover, since $\tau_1(T_h)=0$,  \eqref{eq:compositions} shows that the $\tau_2(\gamma)$ term will not contribute to the composition, and can thus be disregarded. Since $\tau_3(T_h)=0$, the only nontrivial contributions  come from the mixing of $\tau_1(\gamma)$ and $\tau_2(T_h)$ (cf. Proposition 3.4 of \cite{Morita91}):
\begin{equation}\label{eq:tau3recipe}
\begin{split}
\tau_3(\gamma,h)&=   [\tau_1(\gamma),{}^\gamma\tau_2(T_h)]\\
&=(1\otimes 1 \otimes \tau_1(\gamma)+ 1 \otimes \tau_1(\gamma)\otimes 1 + \tau_1(\gamma) \otimes  1\otimes  1)\circ  (^\gamma\tau_2(T_h)) \\
&- (1\otimes {}^\gamma \tau_2(T_h)+ {}^\gamma \tau_2(T_h)\otimes 1)\circ \tau_1(\gamma),
\end{split}
\end{equation}
where $^\gamma\tau_2$ denotes $|\varphi_\gamma| (\tau_2)$ for $\varphi_\gamma\in\M1g$ represented by $\gamma$.  Note that  $^\gamma\tau_2(T_h)$ is given by \eqref{eq:tau2formula} and the action $|\varphi_\gamma| $ of $\gamma$ on $H$.  

The computational thrust of this approach is the determination of $\tau_1(\gamma)$, which
can be done either by hand or on the computer.  Moreover,  since $\tau_1(\varphi)=0$, for every $\varphi\in \mathcal{K}_{g,1}$ (again by \eqref{eq:compositions}), we see that $\tau_3$ is actually a (non-crossed) \emph{homomorphism} on $\mathcal{K}_{g,1}$; of course,  it follows that the values of $\tau_3$ on generators determine it uniquely.  Using the formulae, we have proved:

\begin{corollary}
The third Johnson homomorphism $\tau _3$ is a homomorphism on ${\mathcal K}_{g,1}$.
\end{corollary}

A similar (yet obviously more complicated) approach can be taken to calculate $\tau_4$ on  $\mathcal{K}_{g,1}$.  We sketch the relevant aspects of the calculation of $\tau_4$ on a generator $(\gamma,h)$ below but first comment that 
while $\tau_4$ is no longer a homomorphism on $\mathcal{K}_{g,1}$, \eqref{eq:compositions} 
shows that it is one ``crossed'' only by $\tau_2$ terms since $\tau_1$ vanishes there.   Thus once $\tau_4$ is known for our generating set,  we can easily calculate $\tau_4$ for any product of these generators.  

Since $\tau_1(T_h)=\tau_3(T_h)=0$, the only mixing in $\tau_{\leq 4}(\gamma)$ and $\tau_{\leq 4}(T_h)$ occurs between their respective $\tau_2$ terms.  Thus, we find
\[ 
\tau_4(\gamma,h)=  {}^\gamma \tau_4(T_h)+  [\tau_2(\gamma),{}^\gamma\tau_2(T_h)]
\]
with  $\tau_2(\gamma)$  quite calculable  via computer and $\tau_2(T_h)$  given by \eqref{eq:tau2formula}.  

Now, the calculation of  $\tau_4(T_\partial)$ involves contributions from \eqref{eq:t4b0} as well as from mixing of lower degree terms.  Since $\tau_1(W)=0$ for every Whitehead move $W$ of the sequence, \eqref{eq:compositions} shows that only mixing of $\tau_2(W)$ terms need be considered,  and these have already been determined above.  Finally, 
by Proposition \ref{prop:tail}, each contribution coming from  \eqref{eq:t4b0}  is of the form
\[
 \tG_4 = \frac{1}{6^4}\big(c\otimes ([a,R_t] -6 [ Q_a,\omega ]) -a\otimes([c,R_t] -6 [ Q_c,\omega]) \big),
\]
which again should be amenable to computer calculation.

\subsection{An explicit example}\label{sec63}
 To conclude, we employ the results of the previous section to
compute the value of $\tau_3$ for (right-handed) Dehn twists on certain specific separating curves.  In particular, we focus on elements of the form $(\varphi,h)=\varphi T_h \varphi^{-1}$ with $\varphi\in\M1g$ one of the  Lickorish generators $\lambda_i$, $\mu_i$ (for  $h=1,\ldots ,g$) or $\nu_i$ (for $h=1,\ldots ,g-1$).  We refer to \cite{Morita93} for both notation and for comparison of results.  

We begin by calculating $\tau_1$ for each of Lickorish's generators.  First of all, it obvious that $\tau_1({\lambda_1})=\tau_1({\mu_1})=0$ for a surface of genus one.  By the groupoid stability property of bordered fatgraphs, this must also be the case for surfaces of higher genus $g$.  Thus, we have $\tau_1({\lambda_i})=\tau_1({\mu_j})=0$ for all $i,j=1,\ldots,g$.  Note that this differs from the lift of Morita \cite{Morita93}.  Also, one can directly use the action of $\nu_i$ on the generators $\{\mb{u}_i,\mb{v}_i\}$ of $\pi_1$ (see \cite{Morita93}) and the fatgraph Magnus expansion to determine that 
\begin{multline*}
\tau_1(\nu_i)=\frac12 (u_i+u_{i+1})\wedge v_i \wedge v_{i+1}\\
=\frac12\big(  (u_i+u_{i+1})\otimes [ v_i, v_{i+1}]+v_i\otimes[v_{i+1},u_i+u_{i+1}]+v_{i+1}\otimes[u_i+u_{i+1},v_i] \big)
\end{multline*}
which does match the lift of Morita up to sign.  

Now, as an immediate consequence of $\tau_1({\lambda_i})=\tau_1({\mu_j})=0$ and \eqref{eq:tau3recipe}, we find that $\tau_3(\lambda_i,h)=   \tau_3(\mu_j,h)= 0$, which simply reflects the fact that the twists $\lambda_i$ and $\mu_j$  commute with any standard separating Dehn twist $T_h$.  Similarly, $\tau_3(\nu_i,h)=0$ for $i\neq h$, and we next turn towards calculating 
\[
\begin{split}
\tau_3(\nu_h,h)&=[\tau_1(\nu_h),{}^{\nu_h}\tau_2(T_h)]\\
\end{split}
\]

We first note that the  action $|\nu_h|$ on $H$ is given by $u_h\mapsto u_h+v_{h+1}-v_h$ and $u_h\mapsto u_h+v_{h+1}-v_h$ while other basis elements are fixed.  If we let  $\omega_h=\sum_{i=1}^h [u_i,v_i]$ and define $\omp$ by $\omp= |\nu_h|(\omega_h)$, we see that 
\[
\omp=|\nu_h|(\omega_h)=\omega_h- [v_h,v_{h+1}],
\]
so that after some cancellations, we have
\begin{multline*}
{}^{\nu_h}\tau_2(T_h)=|\nu_h|(\sum_{i=1}^h u_i\otimes[v_i,\omega_h]-v_i\otimes[u_i,\omega_h])\\
= \sum_{i=1}^h (u_i\otimes[v_i,\omp]-v_i\otimes [u_i,\omp])
+ v_{h+1}\otimes[v_h,\omp]-v_h\otimes[v_{h+1},\omp] .
\end{multline*}

Define $\alpha=2\tau_1(\nu_h)(\omp)$, so that by 
\[
2(\tau_1(\nu_h)\otimes 1 + 1 \otimes \tau_1(\nu_h))(\omega_h)
=[u_h,[v_h,v_{h+1}]]-[v_h,[u_h+u_{h+1},v_{h+1}]],
\]
we have 
\begin{multline*}
\alpha=2(\tau_1(\nu_h)\otimes 1 + 1 \otimes \tau_1(\nu_h))(\omp)\\
=-[v_{h+1},[u_h,v_h]]-[v_h,[u_{h+1},v_{h+1}]-[v_h,[v_h,v_{h+1}]]+[v_{h+1},[v_h,v_{h+1}]].\\
\end{multline*}

Using this notation, we compute  $2\tau_1(\nu_h)\circ {}^{\nu_h}\tau_2(T_h)$:
\[
\begin{split}
&\big(  (u_h+u_{h+1})\otimes [ v_h, v_{h+1}] +v_h\otimes [v_{h+1} ,u_h+u_{h+1} ] + v_{h+1}\otimes[u_h+u_{h+1} ,v_h ]  \big)  \\
&\circ  \big( \sum_{i=1}^h ( u_i\otimes[v_i,\omp]-v_i\otimes [u_i,\omp])+v_{h+1}\otimes[v_h,\omp]-v_h\otimes[v_{h+1},\omp] \big)\\
&= \sum_{i=1}^h (u_i\otimes[v_i,\alpha]
-v_i\otimes [u_i,\alpha])\\
&+u_h\otimes [[v_h,v_{h+1}],\omp]+v_h\otimes [[v_{h+1},u_h+u_{h+1}],\omp]\\
&+v_{h+1}\otimes([v_h,\alpha]+[[v_h,v_{h+1}],\omp])
 -v_h\otimes( [v_{h+1}, \alpha]+[[v_h,v_{h+1}],\omp]) .
\end{split}
\]
 
Similarly, we  compute  ${}^{\nu_h}\tau_2(T_h)\circ 2\tau_1(\nu_h)$:
 \[
 \begin{split}
& \big( \sum_{i=1}^h (u_i\otimes[v_i,\omp]-v_i\otimes [u_i,\omp])+v_{h+1}\otimes[v_h,\omp]-v_h\otimes[v_{h+1},\omp] \big)\\
  &\circ \big(  (u_h+u_{h+1})\otimes [ v_h, v_{h+1}] +v_h\otimes [v_{h+1} ,u_h+u_{h+1} ] + v_{h+1}\otimes[u_h+u_{h+1} ,v_h ] \big)\\
  &=-(u_h+u_{h+1})\otimes [v_{h+1},[v_h,\omp]]    +v_h\otimes [v_{h+1} ,[u_h,\omp] ] \\
  &- v_{h+1}\otimes [v_h,[u_h,\omp]]+ v_{h+1}\otimes [u_h+u_{h+1}, [v_h,\omp]]\\
  &+v_h\otimes [v_{h+1} ,[v_{h+1}-v_h,\omp]] - v_{h+1}\otimes[v_h,[v_{h+1}-v_h,\omp]  ] .\\
 \end{split}
 \]
 
 Combining these, we obtain the expression
\[
\begin{split}
2\tau_3(\nu_h,h) &= \sum_{i=1}^h (u_i\otimes[v_i,\alpha] -v_i\otimes [u_i, \alpha])\\
+u_h&\otimes [v_h,[v_{h+1},\omp]]  +u_{h+1}\otimes  [v_{h+1},[v_h,\omp]]   \\
+v_h&\otimes  ( - [v_{h+1}, \alpha] -[[v_h,v_{h+1}],\omp]+ [[v_{h+1},u_h+u_{h+1}],\omp] \\
&- [v_{h+1} ,[u_h,\omp] ] -[v_{h+1} ,[v_{h+1}-v_h,\omp]] )\\
+ v_{h+1}&\otimes ( [v_h,\alpha]+[[v_h,v_{h+1}],\omp] + [v_h,[u_h,\omp]]\\
&- [u_h+u_{h+1}, [v_h,\omp]]+[v_h,[v_{h+1}-v_h,\omp]  ] ).\\
\end{split}
\]
  
\bibliographystyle{amsplain}

\end{document}